\documentclass[11pt]{amsart}

\evensidemargin\oddsidemargin

\usepackage{amsmath}
\usepackage{amsfonts}
\usepackage{amssymb}

\usepackage{filecontents}

\usepackage[numbers]{natbib}  

\usepackage{hyperref}
\hypersetup{colorlinks,linkcolor={red},citecolor={blue},urlcolor={blue}}

\newtheorem{theorem}{Theorem}[section]
\newtheorem{lemma}[theorem]{Lemma}
\newtheorem{proposition}[theorem]{Proposition}
\newtheorem{corollary}[theorem]{Corollary}

\theoremstyle{definition}

\newtheorem{remark}[theorem]{Remark}

\numberwithin{equation}{section}

\newcommand{\CC}{\mathbb C}

\newcommand{\NN}{\mathbb N}

\newcommand{\ZZ}{\mathbb Z}

\def\A-Lift{\operatorname{A-Lift}}

\newenvironment{psmallmatrix}
  {\left(\begin{smallmatrix}}
  {\end{smallmatrix}\right)}

\begin{document}

\title{On weak Jacobi forms of rank two}

\author{Haowu Wang}

\address{Max-Planck-Institut f\"{u}r Mathematik, Vivatsgasse 7, 53111 Bonn, Germany}

\email{haowu.wangmath@gmail.com}

\author{Brandon Williams}

\address{Lehrstuhl A für Mathematik, RWTH Aachen, 52056 Aachen, Germany}

\email{brandon.williams@matha.rwth-aachen.de}

\subjclass[2010]{11F50}

\date{\today}

\keywords{Jacobi forms, lattices of rank two}

\begin{abstract} 
We study a ring of weak Jacobi forms indexed by integral lattices of rank two. We find an explicit finite set of generators of this ring and give a dimension formula for weak Jacobi forms of rank two lattice index.
\end{abstract}

\maketitle

\section{Introduction}

In this paper we will give structure results for weak Jacobi forms indexed by integral lattices of rank two.  These are two-variable analogues of the weak Jacobi forms introduced by Eichler--Zagier \cite{EZ} that occur in the Fourier-Jacobi expansions of modular forms on $\mathrm{O}(2, 4)$ \cite{G}, on $\mathrm{U}(2, 2)$ \cite{H} and on $\mathrm{Sp}(6)$. Our object of study is the graded ring $$\mathcal{J} := \bigoplus_{a, b, c = 0}^{\infty} \bigoplus_{k \in \mathbb{Z}} J_{k, \begin{psmallmatrix} a+b & b \\ b & c+b \end{psmallmatrix}}^w,$$ where $J_{k, M}^w$ is the space of weak Jacobi forms of weight $k$ and whose index is the lattice with Gram matrix $M$. (We do not require $M$ to be positive-definite or have even diagonal.) The main results are:

\begin{theorem}\label{th:main1} The graded ring $\mathcal{J}$ above is finitely generated. Every weak Jacobi form $\varphi(\tau, z, w)$ in the even subring $$\mathcal{J}_{0} = J^w_{2*, \begin{psmallmatrix} 2* & 2* \\ 2* & 2* \end{psmallmatrix}}$$ can be written as a polynomial in the rank one forms $$f(\tau, z), \; f(\tau, w), \; f(\tau, z+w), \; f \in \{\phi_{-2, 1}, \phi_{0, 1}, \phi_{-1, 2}\},$$ where $\phi_{-2, 1}$, $\phi_{0, 1}$, $\phi_{-1, 2}$ are the generators of the ring of scalar-index weak Jacobi forms of Eichler--Zagier \cite{EZ}. The full ring $\mathcal{J}$ is generated by $9$ rank-one weak Jacobi forms and by $8$ rank-two weak Jacobi forms whose indices are lattices of discriminant $3, 5, 8$.
\end{theorem}

As an application, we will show that all graded rings of the form $\mathcal{J}_L = \bigoplus_{n = 0}^{\infty} J_{*, L(n)}^w$ are finitely generated (see Corollary \ref{cor:ring}). Here $L$ is an integral lattice of rank two and $L(n)$ is that lattice with its bilinear form multiplied by $n$. 

\begin{theorem}\label{th:main2} Let $L$ be an integral lattice of rank two and let $\vartheta$ be the theta function $$\vartheta(\tau, z) = q^{1/8} \zeta^{1/2} \sum_{n = -\infty}^{\infty} q^{n(n+1)/2} (-\zeta)^n, \; q = e^{2\pi i \tau}, \; \zeta = e^{2\pi i z}.$$ Every weak Jacobi form of index $L$ can be expressed as a rational function in the forms $$\vartheta^{(k)}(\tau, \lambda z), \; k \leq 5,$$ where $\lambda \in L$ and where $\vartheta^{(k)}$ is the $k$-th derivative with respect to $z$.
\end{theorem}

\begin{theorem}\label{th:main3} Let $L$ be an integral lattice of rank two and choose a Gram matrix of the form $\begin{psmallmatrix} a+b & b \\ b & c+b \end{psmallmatrix}$ with $a, b, c \in \mathbb{N}_0$. Let $P_a$, $Q_a$ denote the Laurent polynomials $$P_a(t) = t^{-a} + \frac {1 - t^{1-a}}{1 - t^{-1}} = t^{-a} + t^{2-a} + ... + 1, \; a \ge 2$$ and $$Q_a(t) = \frac{1 - t^{-a}}{1 - t^{-1}} = t^{1-a} + t^{2-a} + ... + 1, \; a \ge 1,$$ and define $P_0(t) = 1, P_1(t) = t^{-1}$ and $Q_{-1}(t) = Q_0(t) = 0$. Then the $\mathbb{C}[E_4, E_6]$-module $J_{*, L}^w$ of weak Jacobi forms of index $L$ is free on generators $\varphi_1,...,\varphi_n$ of weights $k_1,...,k_n$ that are determined by 
\begin{align*} & t^{k_1} + ... + t^{k_n} \\ =& P_a(t)P_b(t)P_c(t) + Q_a(t) Q_b(t) Q_c(t) + (2t^{-1} - 1) Q_{a-1}(t) Q_{b-1}(t) Q_{c-1}(t) \\ &- t^{-1} \Big(Q_{a}(t)Q_{b-1}(t)Q_{c-1}(t) + Q_{a-1}(t)Q_b(t)Q_{c-1}(t) + Q_{a-1}(t)Q_{b-1}(t)Q_c(t)\Big) \\  &+ \begin{cases} t^{1-a-b-c}: & abc \ne 0; \\ 0: & abc = 0. \end{cases}
\end{align*}
\end{theorem}

This project was motivated partly by physical applications. Weak Jacobi forms of rank two have applications to $6D$ superconformal field theories (cf. \cite{CHS, DL, DGHKK,GHKK, HLV}) where they occur as elliptic genera and partition functions of strings. In \cite{HLV} Babak, Lockhart and Vafa conjectured a fusion rule from $n$ pairs of E-strings to $n$ heterotic strings, which yields a new approach to explicit expressions for the elliptic genus of $n$ E-strings, and is equivalent to certain non-trivial identities among weak Jacobi forms associated to lattices of type $E_8(n)\oplus M_n$ where $M_n$ is an integral lattice of rank $2$. When $n=1,2$, the lattice $M_n$ is diagonal so Jacobi forms of index $M_n$ reduce to the scalar-index Jacobi forms of Eichler--Zagier \cite{EZ}, and in these cases the conjecture was proved in \cite{HLV}. In general, the lattice $M_n$ is not diagonal (for instance, $M_3$ is $\mathbb{Z}^2$ with Gram matrix $\begin{psmallmatrix} 4 & 1 \\ 1 & 4 \end{psmallmatrix}$), and there is no general structure theorem for weak Jacobi forms of rank two in the literature. The case $n = 3$ was proved in \cite{CHS} by working in a diagonal sublattice. This argument also appears in \cite{GHKK}. Passing to a diagonal sublattice creates many redundant generators which causes the computation to be quite difficult; we hope the structure results proved here will simplify such arguments.

The layout of this paper is as follows. In \S \ref{sec:pre} we review some properties of Jacobi forms and show that the ring of weak Jacobi forms associated to the direct sum $L\oplus M$ of two integral lattices is generated by tensor products of the generators of weak Jacobi forms of indices $L$ and $M$ (see Theorem \ref{thm:diag}). In \S \ref{sec:matrix} we define Jacobi forms of matrix index.  \S \ref{sec:construction} is devoted to the construction of generators. In \S \ref{sec:even} we determine the ring of weak Jacobi forms of even weight associated to even lattices, proving the first part of Theorem \ref{th:main1} (cf. Theorem \ref{thm:even}). The second part of Theorem \ref{th:main1} (i.e. Theorem \ref{thm:gens}) is proved in \S \ref{sec:general}. Theorem \ref{th:main2} follows from Theorem \ref{thm:gens} and our construction of the generators. In \S \ref{sec:weight} we study the Hilbert series of weak Jacobi forms and prove Theorem \ref{th:main3} as a corollary of Theorem \ref{thm:hilbert}.

\section{Preliminaries}\label{sec:pre}

Let $L$ be an integral lattice with positive semidefinite bilinear form $\langle -, - \rangle : L \otimes L \rightarrow \mathbb{Z}$ and quadratic form $$Q : L \rightarrow \frac{1}{2}\mathbb{Z}, \; \; Q(x) = \langle x, x \rangle / 2.$$ For every $N \in \mathbb{N}$ we let $L(N)$ denote the lattice $L$ with quadratic form $N \cdot Q$. 

A \emph{weakly holomorphic Jacobi form} of weight $k \in \mathbb{Z}$ and index $L$ is a holomorphic function $$f : \mathbb{H} \times (L \otimes \mathbb{C}) \longrightarrow \mathbb{C}$$ that satisfies $$f\Big( \frac{a \tau + b}{c \tau + d}, \frac{\mathfrak{z}}{c \tau + d} \Big) = (c \tau + d)^k \exp \Big( \frac{2\pi i c}{c \tau + d} Q(\mathfrak{z}) \Big) f(\tau, \mathfrak{z})$$ and $$f(\tau, \mathfrak{z} + \lambda \tau + \mu) = (-1)^{Q(\lambda)+ Q(\mu)} \exp \Big( -2\pi i \tau Q(\lambda) - 2\pi i \langle \lambda, \mathfrak{z} \rangle \Big) f(\tau, \mathfrak{z})$$ for all $\begin{psmallmatrix} a & b \\ c & d \end{psmallmatrix} \in \mathrm{SL}_2(\mathbb{Z})$ and $\lambda, \mu \in L$, and whose Fourier expansion as a function of $\tau$, $$f(\tau, \mathfrak{z}) = \sum_{n \gg -\infty} c(n, \mathfrak{z}) q^n, \; \; q = e^{2\pi i \tau}$$ involves only finitely many negative powers of $q$. We call $f$ a \emph{weak Jacobi form} if its Fourier series is supported only on nonnegative exponents: $$f(\tau, \mathfrak{z}) = \sum_{n=0}^{\infty} c(n, \mathfrak{z}) q^n,$$ and a \emph{holomorphic Jacobi form} if every function $f(\tau,\lambda \tau + \mu)$ with $\lambda, \mu \in \frac{1}{N}L$, $N \in \mathbb{N}$ is a holomorphic modular form (of level $\Gamma(N)$).

The spaces of weakly holomorphic Jacobi forms, weak Jacobi forms and holomorphic Jacobi forms of weight $k$ and index $L$ are labelled $$J_{k, L}^!, \; J_{k, L}^w, \; J_{k, L}$$ respectively.

This definition extends the standard notion of Jacobi forms of lattice index by allowing odd lattices (i.e. vectors may have half-integral norm) and by allowing degenerate lattices. Neither of these enlarges the class of Jacobi forms significantly: if $f(\tau, \mathfrak{z})$ has index $L$, then $f(\tau, 2 \mathfrak{z})$ has even index $L(4)$; and if $L$ is a degenerate lattice with kernel $$\mathrm{ker}(L) = \{y \in L: \; \langle x, y \rangle = 0 \; \text{for all} \; x \in L\}$$ then we have an identification $$\pi^* : J^w_{*, L / \mathrm{ker}(L)} \stackrel{\sim}{\longrightarrow} J^w_{*, L}, \; \; f \mapsto \pi^*f(\tau, \mathfrak{z}) := f(\tau, \pi(\mathfrak{z})),$$ where $\pi : L \otimes \mathbb{C} \rightarrow (L / \mathrm{ker}(L)) \otimes \mathbb{C}$ is the quotient map. 

When $\mathrm{rank}(L) \le 1$, (weak) Jacobi forms of matrix index $M = (m)$ are the same as (weak) Jacobi forms of index $m/2$ as defined by Eichler--Zagier \cite{EZ}, and will be referred to as such. In particular, by \emph{Jacobi forms of index $0$} we mean functions $f(\tau, z)$ that are constant in $z$ and modular forms of level $\mathrm{SL}_2(\mathbb{Z})$ in $\tau$. The fundamental example of a weak Jacobi form of half-integral index is the form (see \cite{G99}) $$\phi_{-1, 1/2}(\tau, z) = \frac{\vartheta(\tau, z)}{\eta^3(\tau)}$$ of weight $-1$ and index $1/2$, where $\vartheta$ is the \emph{Jacobi theta function} as defined in Theorem \ref{th:main2} and $\eta(\tau) = q^{1/24} \prod_{n=1}^{\infty} (1 - q^n)$ is the Dedekind eta function, such that $\eta^3(\tau) = \vartheta'(\tau, 0)$ by Jacobi's identity. More generally:

\begin{proposition}\label{prop:rankone} The graded ring of weak Jacobi forms of rank one and half-integer index is generated by the Eisenstein series $E_4, E_6$ and by the weak Jacobi forms $$\phi_{-1, 1/2}(\tau, z) = \frac{\vartheta(\tau, z)}{\eta^3(\tau)} = (-\zeta^{-1/2} + \zeta^{1/2}) + (\zeta^{-3/2} - 3 \zeta^{-1/2} + 3 \zeta^{1/2} - \zeta^{3/2})q + O(q^2),$$ $$\phi_{0, 1}(\tau, z) = \frac{\vartheta^2(\tau, z)}{\eta^6(\tau)} \cdot \frac{3}{\pi^2} \wp(\tau, z) = (\zeta^{-1} + 10 + \zeta) + (10 \zeta^{-2} - 64 \zeta^{-1} + 108 - 64\zeta + 10\zeta^2)q + O(q^2),$$ $$\phi_{0, 3/2}(\tau, z) = \frac{\vartheta(\tau, 2z)}{\vartheta(\tau, z)} = (\zeta^{-1/2} + \zeta^{1/2}) + (-\zeta^{-5/2} + \zeta^{-1/2} + \zeta^{1/2} - \zeta^{5/2})q + O(q^2).$$ There is a decomposition $$J^w_{*, */2} = \mathbb{C}[\phi_{-1, 1/2}, \phi_{0, 1}] \oplus \phi_{0, 3/2} \mathbb{C}[\phi_{-1, 1/2}, \phi_{0, 1}].$$
\end{proposition}

Here $$\wp(\tau, z) = \frac{1}{z^2} + \sum_{\omega \in (\mathbb{Z} + \tau \mathbb{Z}) \backslash \{0\}} \Big( \frac{1}{(z + \omega)^2} - \frac{1}{\omega^2} \Big)$$ is the Weierstrass elliptic function. The subring of integer-index weak Jacobi forms is then generated by the forms $$\phi_{-2, 1} := \phi_{-1, 1/2}^2, \; \phi_{0, 1}, \; \text{and} \; \phi_{-1, 2} := \phi_{-1, 1/2}\cdot \phi_{0, 3/2},$$ cf. \cite{EZ}. Conversely, Proposition \ref{prop:rankone} can be easily derived from the structure theorem of \cite{EZ}.\\

It will be important at many points that the theta function $\vartheta(\tau, z)$ (and therefore also $\phi_{-1, 1/2}$) for fixed $\tau$ has simple zeros in the lattice points $z \in \mathbb{Z} + \tau \mathbb{Z}$ and nowhere else. This follows immediately from the Jacobi triple product $$\vartheta(\tau, z) = q^{1/8} \zeta^{1/2} \prod_{n=1}^{\infty} (1 - q^n)(1 - q^n \zeta) (1 - q^{n-1} \zeta^{-1})$$ or it can be proved more directly.

For any fixed index $L$, the Jacobi forms $J_{*, L}$ and weak Jacobi forms $J_{*, L}^w$ of index $L$ and all weights can be viewed as modules over the ring $\mathbb{C}[E_4, E_6]$ of modular forms. Abstractly the structure of these modules is well-known, but we include a proof for convenience:

\begin{proposition}\label{prop:free} $J_{*, L}$ and $J_{*, L}^w$ are free $\mathbb{C}[E_4, E_6]$-modules. If $L$ is positive-definite then both modules have rank $\mathrm{det}(L)$.
\end{proposition}
More generally, for any congruence subgroup $\Gamma \le \mathrm{SL}_2(\mathbb{Z})$, the $M_*(\Gamma)$-modules $J_{*, L}$ and $J_{*, L}^w$ are free on the same $\mathrm{det}(L)$ generators.
\begin{proof} To see that $J_{*, L}$ and $J_{*, L}^w$ are free, one can adapt the proof of Theorem 8.4 of \cite{EZ} (which, as remarked there, applies to a wide class of modules over $M_*(\Gamma)$; the only essential ingredients are that $J_{k, L}$ and $J_{k, L}^w$ are always finite-dimensional, and are zero for sufficiently low $k$). Since $J_{*, M} \subseteq J_{*, M}^w$ and $\Delta^r J_{*, M}^w \subseteq J_{*, M}$ for all large enough $r \in \mathbb{N}$, or more precisely whenever $$r \ge \max_{\gamma \in L'/L} \min_{x \in \gamma + L} Q(x),$$ the modules $J_{*, M}$ and $J_{*, M}^w$ have the same rank. 

\smallskip

When $L$ is an even index we can identify $J_{*, L}$ with the module of vector-valued modular forms for the Weil representation attached to $L$. Then the Riemann-Roch theorem (through the formula of section 4 of \cite{B}) implies that, for weights $k > 2 + \mathrm{rank}(L) / 2$, $$\mathrm{dim} \, J_{k + 12, L} = \mathrm{dim}\, J_{k, L} + \begin{cases} \mathrm{dim}\, \mathrm{span}(e_{\gamma} + e_{-\gamma}: \; \gamma \in L'/L): & k \equiv 0 \, (2); \\ \mathrm{dim}\, \mathrm{span}(e_{\gamma} - e_{-\gamma}: \; \gamma \in L'/L): & k \equiv 1 \, (2); \end{cases}$$ where $e_{\gamma} \in \mathbb{C}[L'/L]$ is the group ring vector attached to $\gamma \in L'/L$. This implies that the even-weight and odd-weight submodules are free of rank $$\mathrm{rank}\, J_{2*, L} = \mathrm{rank}\, J_{2*, L}^w = \mathrm{dim}\, \mathrm{span}(e_{\gamma} + e_{-\gamma})$$ and $$\mathrm{rank}\, J_{2*+1, L} = \mathrm{rank} \, J_{2*+1, L}^w = \mathrm{dim}\, \mathrm{span}(e_{\gamma} - e_{-\gamma}),$$ and we immediately obtain the full rank $$\mathrm{rank}\, J_{*, L} = \mathrm{rank}\, J_{2*, L} + \mathrm{rank}\, J_{2*+1, L} = \mathrm{det}(L).$$

This rank also has the following interpretation. The $q^0$-term $c(0, \mathfrak{z})$ of a weak Jacobi form of weight $k$ and even index $L$ is a Laurent polynomial $$c(0, \mathfrak{z}) = \sum_{r \in L'} c(0, r) \zeta^r, \; \; \zeta^r = e^{2\pi i \langle r, \mathfrak{z} \rangle}$$ where $c(0, r)$ may be nonzero only if $r$ has minimal norm among all vectors in its coset $r + L$, and where $c(0, r) = (-1)^k c(0, -r)$. That $J_{*, L}^w$ has rank $\mathrm{det}(L)$ is equivalent to the fact that every Laurent polynomial satisfying these two conditions actually occurs as the $q^0$-term of a weak Jacobi form (of some weight); in other words, if $V$ is the space spanned by the Fourier expansions of weak Jacobi forms (of all weights), then the map $$V / \Delta V \stackrel{\sim}{\longrightarrow} \mathrm{span}\Big(\zeta^r: \; r \in L' \; \text{of minimal norm in its coset}\Big)$$ sending $\sum_{n, r} c(n, r) q^n \zeta^r$ to its $q^0$-term is an isomorphism of vector spaces. 

\smallskip

Now if $L$ is an odd lattice, we obtain $$\mathrm{rank}\, J_{*, L} = \mathrm{rank}\, J_{*, L}^w = \mathrm{det}(L)$$ by identifying the Fourier expansions of forms $\{f(\tau, 2 \mathfrak{z}): \; f \in J_{*, L}^w\}$ with the subspace of $V / (\Delta V)$ of $q^0$-terms of Jacobi forms of index $L(4)$ satisfying $$c(0, r) = 0 \; \text{unless} \; r/2 \in L(4)';$$ this is seen to be a $\mathrm{det}(L)$-dimensional space, and the claim follows.
\end{proof}

\begin{corollary} For every integral lattice $L$, there is a polynomial $P_L$ and a Laurent polynomial $P_L^{w}$ such that $$\sum_{k=0}^{\infty} \mathrm{dim}\, J_{k, L} t^k = \frac{P_L(t)}{(1 - t^4)(1 - t^6)}, \; \; \sum_{k=-\infty}^{\infty} \mathrm{dim}\, J_{k, L}^w t^k = \frac{P_L^{w}(t)}{(1 - t^4) (1 - t^6)}.$$ They have the form $$P_L(t) = t^{k_1} + ... + t^{k_n}, \; \; P_L^{w}(t) = t^{\ell_1} + ... + t^{\ell_n},$$ where $k_1,...,k_n$ are the weights of generators of the $\mathbb{C}[E_4, E_6]$-module $J_{*, L}$ and where $\ell_1, ..., \ell_n$ are the weights of generators of the $\mathbb{C}[E_4, E_6]$-module $J_{*, L}^w$.
\end{corollary}

The following description of weakly holomorphic and weak Jacobi forms associated to a direct sum of lattices has no analogue for holomorphic Jacobi forms.

\begin{theorem}\label{thm:diag} For any integral lattices $L, M$ there are isomorphisms of $\mathbb{C}[E_4, E_6]$-modules $$J_{*, L}^! \otimes J_{*, M}^! \stackrel{\sim}{\longrightarrow} J_{*, L \oplus M}^!,$$
$$J_{*, L}^w \otimes J_{*, M}^w \stackrel{\sim}{\longrightarrow} J_{*, L \oplus M}^w,$$
given by the \emph{direct product} $$(f, g) \mapsto (f \otimes g)(\tau, \mathfrak{z}_L, \mathfrak{z}_M):= f(\tau, \mathfrak{z}_L) g(\tau, \mathfrak{z}_M),$$
where $\mathfrak{z}_L \in L \otimes \mathbb{C}$, $\mathfrak{z}_M \in M \otimes \mathbb{C}$, and $(\mathfrak{z}_L, \mathfrak{z}_M) \in (L \oplus M) \otimes \mathbb{C}$.
In terms of the Laurent polynomials $P^w$ this implies $$P_{L \oplus M}^w = P_L^w \cdot P_M^w.$$
\end{theorem}
\begin{proof} We can assume without loss of generality that $L$ and $M$ are positive-definite, because $J_{*, L}^!$, $J_{*, L}^w$ can otherwise be identified with $J_{*, L/\mathrm{ker}(L)}^!$, $J_{*, L/\mathrm{ker}(L)}^w$ and because $$(L \oplus M) / \mathrm{ker}(L \oplus M)  = (L \oplus M) / (\mathrm{ker}(L) \oplus \mathrm{ker}(M)) \cong (L / \mathrm{ker}(L)) \oplus (M / \mathrm{ker}(M)).$$ Similarly, we can assume without loss of generality that $L$ and $M$ are even lattices; otherwise, substitute $\mathfrak{z} \mapsto 2 \mathfrak{z}$. 

\bigskip

Suppose first that $M = (2m)$ has rank $1$. To any weakly holomorphic Jacobi form $h \in J_{k, L \oplus M}^!$ we associate as in \cite{EZ} a sequence of weakly holomorphic Jacobi forms $h_0, h_1,...$ of index $L$ and weights $k,k+1,...$ as the Taylor coefficients of the function (see \cite[Proposition 1.5]{G99}) $$\tilde h(\tau, \mathfrak{z}_L, z) := e^{mz^2 G_2(\tau)} h(\tau, \mathfrak{z}_L, z) = \sum_{n=0}^{\infty} h_n(\tau, \mathfrak{z}_L) z^n$$ about $0$. Here $$G_2(\tau) = \frac{\pi^2}{3} \Big( 1 - 24 \sum_{n=1}^{\infty} \sigma_1(n) q^n \Big)$$ is the non-normalized Eisenstein series of weight two. 

We split the space $J_{k, L \oplus M}^!$ into its even and odd subspaces (with respect to $M$): $$J_{k, L \oplus M}^{!, \text{even}} = \{h \in J_{k, L \oplus M}^!: \; h(\tau, \mathfrak{z}_L, -z) = h(\tau, \mathfrak{z}_L, z)\},$$ $$J_{k, L \oplus M}^{!, \text{odd}} = \{h \in J_{k, L \oplus M}^!: \; h(\tau, \mathfrak{z}_L, -z) = -h(\tau, \mathfrak{z}_L, z)\}.$$

Then the map $$J_{k, L \oplus M}^{!, \text{even}} \stackrel{\sim}{\longrightarrow} \bigoplus_{n=0}^m J_{k + 2n, L}^!, \; \; h \mapsto (h_0, h_2, ..., h_{2m})$$ is injective, because for fixed $\tau$ and $\mathfrak{z}_L$ the function $h(\tau, \mathfrak{z}_L, z)$ has exactly $2m$ zeros in any fundamental parallelogram for $\mathbb{C} / (\mathbb{Z} + \mathbb{Z}\tau)$, and it is surjective as any sequence $(h_0,...,h_{2m})$ arises (essentially) as the sequence of coefficients of the form $$h(\tau, \mathfrak{z}_L, z) = h_0(\tau, \mathfrak{z}_L) \phi_{0, 1}^m(\tau, z) + h_2(\tau, \mathfrak{z}_L) \phi_{0, 1}^{m-1}(\tau, z) \phi_{-2, 1}(\tau, z) + ... + h_{2m}(\tau, \mathfrak{z}_L) \phi_{-2, 1}^m(\tau, z).$$ Under this map the subspace of weak Jacobi forms $J_{k, L \oplus M}^{w, \text{even}}$ is identified with $\bigoplus_{n=0}^m J_{k+2n, L}^w$. 

It follows that $$J_{k, L \oplus M}^{!, \text{odd}} \stackrel{\sim}{\longrightarrow} \bigoplus_{n=1}^{m-1} J_{k + 2n - 1, L}^! \; \text{and} \; J_{k, L \oplus M}^{w, \text{odd}} \stackrel{\sim}{\longrightarrow} \bigoplus_{n=1}^{m-1} J_{k + 2n - 1, L}^w,$$ $$h \mapsto (h_1,h_3,...,h_{2m-3})$$ are isomorphisms: we can apply the result for even forms to the quotient $h(\tau, \mathfrak{z}_L, z) / \phi_{-1, 2}(\tau, z)$. This remains a weakly-holomorphic Jacobi form, and it is a weak Jacobi form if $h$ was, because any $h \in J_{k, L \oplus M}^{w, \text{odd}}$ has forced zeros at the $2$-torsion points $z = 0, 1/2, \tau/2, (\tau+1)/2$ where $\phi_{-1, 2}(\tau, z)$ has its simple zeros. In particular, every weak Jacobi form of index $L \oplus M$ can be written as a linear combination of products of weak Jacobi forms $h_i(\tau, \mathfrak{z}_L)$ of index $L$ by weak Jacobi forms $f(\tau, z)$ of rank one. 

\bigskip

In general let $\tilde M$ be a diagonal sublattice of $M$. Repeatedly applying the rank one argument above yields the isomorphism $$\otimes : J_{*, L}^! \otimes J_{*, \tilde M}^! \stackrel{\sim}{\longrightarrow} J_{*, L \oplus \tilde M}^!.$$ Let $R$ denote the averaging map $$R : J_{*, \tilde M}^! \longrightarrow J_{*, M}^!, \; \; h \mapsto  \sum_{\zeta \in \mathcal{H}(M) / \mathcal{H}(\tilde M)} h \Big| \zeta,$$ which extends naturally to the map $J_{*, L \oplus \tilde M}^! \rightarrow J_{*, L \oplus M}^!$ by acting only on the variable $\mathfrak{z}_M$, where $\mathcal{H}(M)$ and $\mathcal{H}(\tilde M)$ are respectively the integral Heisenberg groups of $M$ and $\tilde{M}$ (see \cite{G}). Then we obtain $$J_{*, L \oplus M}^! = R \Big( J_{*, L \oplus \tilde M}^! \Big) = R \Big( J_{*, L}^! \otimes J_{*, \tilde M}^! \Big) = J_{*, L}^! \otimes J_{*, M}^!.$$ This averaging argument requires some care for weak Jacobi forms, since the image of a weak Jacobi form under $R$ is not necessarily a weak Jacobi form. That $J_{*, L \oplus M}^w$ is indeed identified with $J_{*, L}^w \otimes J_{*, M}^w$ under this map follows from the following lemma.
\end{proof}

\begin{lemma} Suppose $f_1,...,f_r \in J_{*, L}^w$ are a $\mathbb{C}[E_4, E_6]$-basis of weak Jacobi forms of index $L$, and suppose $g_1,...,g_r \in J_{*, M}^!$ are weakly holomorphic Jacobi forms for which $$h := f_1 \otimes g_1 + ... + f_r \otimes g_r$$ is a weak Jacobi form. Then each of $g_1,...,g_r$ is a weak Jacobi form.
\end{lemma}
\begin{proof} Let $d$ be the minimal exponent for which some $g_i$ has a nonzero coefficient of $q^d$, and suppose $d < 0$. Write $$f_i = \sum_{n=0}^{\infty} a(f_i, n) q^n, \; \; g_i = \sum_{n=d}^{\infty} b(g_i, n) q^n,$$ where $a(f_i, n)$ and $b(g_i, n)$ are Laurent polynomials in independent variables $\mathfrak{z}_L$ and $\mathfrak{z}_M$, such that $$h(\tau, \mathfrak{z}) = \underbrace{\Big( \sum_{i=1}^r a(f_i, 0)(\mathfrak{z}_L) b(g_i, d)(\mathfrak{z}_M) \Big)}_{=0} q^d + O(q^{d+1}).$$ Since $f_1,...,f_r$ are a basis of $J_{*, L}^w$, their $q^0$-terms are $\mathbb{C}$-linearly independent.  In particular, $b(g_i, d)$ are Laurent polynomials which are zero at every value of $\mathfrak{z}_M$ and therefore identically zero, contradicting the definition of $d$.
\end{proof}

\section{Weak Jacobi forms of matrix index} \label{sec:matrix}

It will be more convenient to fix the underlying group $L = \mathbb{Z}^2$ and consider Jacobi forms indexed by a varying Gram matrix on $L$. Throughout the rest of this paper, an \emph{index matrix} is a matrix of the form $$M = \begin{psmallmatrix} a + b & b \\ b & c + b \end{psmallmatrix} \in \mathbb{Z}^{2 \times 2}, \; \text{with} \; a, b, c \in \mathbb{N}_0.$$ We do not require $M$ to have full rank. Weak Jacobi forms of index $M$ are then written as functions $$f(\tau,  z, w), \; \; \tau \in \mathbb{H}, \; z, w \in \mathbb{C}.$$ We will study the space of all weak Jacobi forms of all weights and all index matrices: $$\mathcal{J} := \bigoplus_{k \in \mathbb{Z}} \bigoplus_{a, b, c=0}^{\infty} J_{k, \begin{psmallmatrix} a + b & b \\ b & c + b \end{psmallmatrix}}^w.$$

It is not hard to see that every rank two lattice $(L, Q)$ is isometric to $\mathbb{Z}^2$ with quadratic form $z^T M z / 2$ for some index matrix $M$. The anharmonic group $$G = \langle \begin{psmallmatrix} 1 & 1 \\ -1 & 0 \end{psmallmatrix}, \begin{psmallmatrix} 0 & 1 \\ 1 & 0 \end{psmallmatrix} \rangle \subseteq \mathrm{PGL}_2(\mathbb{Z})$$ acts on the set of index matrices by conjugation. In particular, the ring $\mathcal{J}$ is preserved under the following changes of coordinates in the elliptic variables: $$(z, w) \mapsto (z+w, -w), \; (w, -z-w), \; (w, z).$$ More precisely, if $f(\tau, z, w)$ has index $\begin{psmallmatrix} a+b & b \\ b & c+b \end{psmallmatrix}$, then $f(\tau, z+w, -w)$ has index $\begin{psmallmatrix} b+a & a \\ a & c+a \end{psmallmatrix}$; and $f(\tau, w, -z-w)$ has index $\begin{psmallmatrix} b+c & c \\ c & a+c \end{psmallmatrix}$; and $f(\tau, w, z)$ has index $\begin{psmallmatrix} c+b & b \\ b & a+b \end{psmallmatrix}$. In other words, if $M$ is written in the form $\begin{psmallmatrix} a+b & b \\ b & c+b \end{psmallmatrix}$ then $G$ acts as the permutations of the indices $(a, b, c)$. If the indices are such that $b \le a \le c$ and we do not allow multiples of the identity then the quadratic form $(a+b)x^2 + 2bxy + (c+b)y^2$ is the \emph{reduced form} in its $\mathrm{SL}_2(\mathbb{Z})$-equivalence class; however it is not convenient to make these restrictions. 

\bigskip

If $f(\tau, z)$ is any (weak) Jacobi form of index $m \in \frac{1}{2}\mathbb{N}$, then we obtain (weak) Jacobi forms of positive semidefinite matrix index by substituting values in the elliptic variable in $f$. In particular, the form $F_1(\tau, z, w) := f(\tau, z)$ has index $\begin{psmallmatrix} 2m & 0 \\ 0 & 0 \end{psmallmatrix}$; $F_2(\tau, z, w) := f(\tau, w)$ has index $\begin{psmallmatrix} 0 & 0 \\ 0 & 2m \end{psmallmatrix}$; and $F_3(\tau, z, w) := f(\tau, z+w)$ has index $\begin{psmallmatrix} 2m & 2m \\ 2m & 2m \end{psmallmatrix}$. Taking products of these forms yields Jacobi forms of (true) rank two. When $f = \phi_{-1, 1/2}$ the forms constructed this way are called \emph{theta blocks} (see \cite{GSZ}). Note that the forms $F_1, F_2, F_3$ are permuted under the anharmonic group.

\section{Construction of generators}\label{sec:construction}

There are various differential operators that can be used to produce Jacobi forms. We mention two here: 

\smallskip

(i)  Suppose $f(\tau, z)$ is a (weak) Jacobi form of weight $k$ and index $m \in \frac{1}{2}\mathbb{N}$. Then $$\frac{(2\pi i)^{-1} f'(\tau, z)}{f(\tau, z)} + 2 m \frac{\mathrm{Im}(z)}{\mathrm{Im}(\tau)}$$ transforms like a Jacobi form of weight $1$ and index $0$ (i.e. an elliptic function). Here, $f'$ denotes differentiation with respect to the $z$ variable. In particular, an expression of the form $$\sum_{i=1}^n \frac{\vartheta'(\tau, \langle \lambda_i, \mathfrak{z} \rangle)}{\vartheta(\tau, \langle \lambda_i, \mathfrak{z} \rangle)}, \; \; \lambda_i \in L$$ is an abelian function (or meromorphic Jacobi form of index zero) if and only if $\sum_i \lambda_i = 0.$ 

\smallskip

(ii) Suppose $f(\tau, \mathfrak{z})$ is a (weak) Jacobi form of weight $k$ and positive-definite index $M$. Then the \emph{Serre derivative} $$\mathcal{S}f(\tau, \mathfrak{z}) := \mathcal{H} f(\tau, \mathfrak{z}) - \frac{3}{\pi^2}\Big( \frac{k}{12} - \frac{\mathrm{rank}(L)}{24} \Big) G_2(\tau)f(\tau, \mathfrak{z})$$ is a (weak) Jacobi form of weight $k+2$ and the same index $M$. Here $\mathcal{H}$ is the heat operator $(2\pi i)^{-1} \partial_{\tau} - (2\pi i)^{-2} \Delta_L / 2$, i.e. on Fourier series it acts as $$\mathcal{H} \Big( \sum_{r \in \mathbb{Z}^n} \sum_{n = 0}^{\infty} c(n, r) q^n \zeta^r \Big) = \sum_{r \in \mathbb{Z}^n} \sum_{n = 0}^{\infty} c(n, r) (n - r^T M^{-1} r / 2) q^n \zeta^r.$$ When $M$ has even diagonal, $\mathcal{S}$ is obtained by applying the Serre derivative componentwise on the corresponding vector-valued modular form (cf. \cite{EZ}). Note that $\mathcal{S} \vartheta = 0$. Using this relation one can express the Serre derivative of a polynomial in the forms $\vartheta(\tau, \lambda z), ..., \vartheta^{(n)}(\tau, \lambda z)$, $\lambda \in L$ as a polynomial in the forms $\vartheta(\tau, \lambda z),..., \vartheta^{(n+2)}(\tau, \lambda z)$.

\bigskip

We can now construct a number of examples of Jacobi forms of rank two which are not simply products of Jacobi forms of rank one. The forms we construct here will be useful in our induction arguments later due to their special values at $w = 0$ and $w = -z$. For now we leave these constructions somewhat unmotivated. Let $q = e^{2\pi i\tau}$, $\zeta = e^{2\pi i z}$ and $\omega = e^{2\pi i w}$.

\begin{lemma} (a corollary of the $A_2$-case of Wirthm\"uller's theorem \cite{W92}) The graded $\mathbb{C}[E_4, E_6]$-module of weak Jacobi forms of matrix index $\begin{psmallmatrix} 2 & 1 \\ 1 & 2 \end{psmallmatrix}$ is free on the basis 
\begin{align*}
\Phi_{-3, \begin{psmallmatrix} 2 & 1 \\ 1 & 2 \end{psmallmatrix}}(\tau, z, w) &= \frac{\vartheta(\tau, z) \vartheta(\tau, w) \vartheta(\tau, z+w)}{\eta^9(\tau)} \\ &= -\zeta^{-1}\omega^{-1} + \zeta^{-1} + \omega^{-1} - \zeta - \omega + \zeta \omega + O(q), \\ 
\Phi_{-2, \begin{psmallmatrix} 2 & 1 \\ 1 & 2 \end{psmallmatrix}}(\tau, z, w) &= \Big( \frac{\vartheta'(\tau, z)}{\vartheta(\tau, z)} + \frac{\vartheta'(\tau, w)}{\vartheta(\tau, w)} - \frac{\vartheta'(\tau, z+w)}{\vartheta(\tau, z+w)} \Big) \Phi_{-3, \begin{psmallmatrix} 2 & 1 \\ 1 & 2 \end{psmallmatrix}} \\ &= \zeta^{-1}\omega^{-1} + \zeta^{-1} + \omega^{-1} - 6 + \zeta + \omega + \zeta \omega + O(q),\\  
\Phi_{0, \begin{psmallmatrix} 2 & 1 \\ 1 & 2 \end{psmallmatrix}}(\tau, z, w) &= -12 \cdot \mathcal{S} \Big(\Phi_{-2, \begin{psmallmatrix} 2 & 1 \\ 1 & 2 \end{psmallmatrix}}\Big) \\ &= \zeta^{-1} \omega^{-1} + \zeta^{-1} + \omega^{-1} + 18 + \zeta + \omega + \zeta \omega + O(q). 
\end{align*} 
\end{lemma}

\begin{lemma} The weak Jacobi form \begin{align*} \Phi_{0, \begin{psmallmatrix} 3 & 2 \\ 2 & 3 \end{psmallmatrix}} &= \frac{5}{2} \mathcal{S} \Big( \frac{\vartheta(\tau, z) \vartheta(\tau, w) \phi_{0, 1}(\tau, z+w)}{\eta^{6}(\tau)} \Big) \\ &= -\frac{1}{2}\zeta^{-3/2} \omega^{-3/2} + \frac{1}{2}\zeta^{-3/2} \omega^{-1/2} + \frac{1}{2}\zeta^{-1/2}\omega^{-3/2} + \frac{11}{2} \zeta^{-1/2} \omega^{-1/2} \\ &\quad \quad + \frac{11}{2} \zeta^{1/2} \omega^{1/2} + \frac{1}{2}\zeta^{1/2} \omega^{3/2} + \frac{1}{2}\zeta^{3/2} \omega^{1/2} - \frac{1}{2}\zeta^{3/2} \omega^{3/2} + O(q) \end{align*} of weight $0$ and index $\begin{psmallmatrix}3 & 2 \\ 2 & 3 \end{psmallmatrix}$ satisfies $$\Phi_{0, \begin{psmallmatrix} 3 & 2 \\ 2 & 3 \end{psmallmatrix}}(\tau, z, -z) = \phi_{0, 1}(\tau, z)$$ and $$\Phi_{0, \begin{psmallmatrix} 3 & 2 \\ 2 & 3 \end{psmallmatrix}}(\tau, z, 0) = \Phi_{0, \begin{psmallmatrix} 3 & 2 \\ 2 & 3 \end{psmallmatrix}}(\tau, 0, z) = 6 \cdot \phi_{0, 3/2}(\tau, z).$$
\end{lemma}

\begin{lemma} The weak Jacobi form \begin{align*} \Phi_{0, \begin{psmallmatrix} 3 & 1 \\ 1 & 3 \end{psmallmatrix}}(\tau, z, w) &= \frac{29}{2} E_4(\tau) \phi_{-1, 1/2}(\tau, z) \phi_{-1, 1/2}(\tau, w) \Phi_{-2, \begin{psmallmatrix} 2 & 1 \\ 1 & 2 \end{psmallmatrix}}(\tau, z, w) \\ &\quad + 48 \cdot \mathcal{S} \Big( \phi_{-1, 1/2}(\tau, z) \phi_{-1, 1/2}(\tau, w) \Phi_{0, \begin{psmallmatrix} 2 & 1 \\ 1 & 2 \end{psmallmatrix}}(\tau, z, w) \Big) \\ &= -\frac{1}{2}\zeta^{-3/2} \omega^{-3/2} + \frac{1}{2} \zeta^{-3/2} \omega^{1/2} + \frac{103}{2} \zeta^{-1/2} \omega^{-1/2} + 20 \zeta^{-1/2} \omega^{1/2} + \frac{1}{2} \zeta^{-1/2} \omega^{3/2} \\ &\quad\quad + \frac{1}{2} \zeta^{1/2} \omega^{-3/2} + 20 \zeta^{1/2} \omega^{-1/2} + \frac{103}{2} \zeta^{1/2} \omega^{1/2} + \frac{1}{2} \zeta^{3/2} \omega^{-1/2} - \frac{1}{2}\zeta^{3/2} \omega^{3/2} + O(q) \end{align*} of weight $0$ and index $\begin{psmallmatrix} 3 & 1 \\ 1 & 3 \end{psmallmatrix}$ satisfies $$\Phi_{0, \begin{psmallmatrix} 3 & 1 \\ 1 & 3 \end{psmallmatrix}}(\tau, z, -z) = \phi_{0, 1}(\tau, z)^2$$ and $$\Phi_{0, \begin{psmallmatrix} 3 & 1 \\ 1 & 3 \end{psmallmatrix}}(\tau, z, 0) = \Phi_{0, \begin{psmallmatrix} 3 & 1 \\ 1 & 3 \end{psmallmatrix}}(\tau, 0, z) = 72 \phi_{0, 3/2}(\tau, z).$$
\end{lemma}

\section{Weak Jacobi forms of even weight}\label{sec:even}

\begin{lemma} \label{lem:surj} Let $M = \begin{psmallmatrix} a & b \\ b & c \end{psmallmatrix}$ be an index matrix. Suppose either 
\begin{itemize}
\item[(i)] $b$ is even, or
\item[(ii)] $b$ is odd and $b < \min(a, c)$.
\end{itemize}
Let $k \in \mathbb{Z}$ be a weight with $k \equiv a + c$ mod $2$. Then the pullback map $$\mathcal{P} : J_{k, M}^w \longrightarrow J_{k, (a + c)/2 - b}^w, \; \; F(\tau, z, w) \mapsto F(\tau, z, -z)$$ is surjective.
\end{lemma}

The condition $b < \min(a, c)$ when $b$ is odd is necessary to rule out the cases $(a, b, c) = (1, 1, 1)$, $(2, 1, 1)$, $(1, 1, 2)$.

\begin{proof} First note that the map $\mathcal{P}$ satisfies $$\mathcal{P}(\phi_{0, 1}(\tau, z + w) \cdot F) = 12 \mathcal{P}(F),$$ since $\phi_{0,1}(\tau, z+w)$ is mapped to $\phi_{0, 1}(\tau, 0) = 12$. In particular, if we have proved that $\mathcal{P}$ is surjective for some index $\begin{psmallmatrix} a & b \\ b & c \end{psmallmatrix}$ then it is also surjective for the index $\begin{psmallmatrix} a + 2 & b + 2 \\ b + 2 & c + 2 \end{psmallmatrix}$. Therefore we may assume that $b=0$ or $b=1$ . 

\smallskip

Suppose $b = 0$. Let $\psi \in J_{k, (a+c)/2}^w$ be any Jacobi form of index $(a+c)/2$. By our assumption on the weight, we can write $\psi$ as a polynomial $$\psi = Q(E_4, E_6, \phi_{-1, 1/2}, \phi_{0, 1}).$$ In each monomial in $Q$, we replace any $c_1$ copies of $\phi_{-1, 1/2}(\tau, z)$ and $c_2$ copies of $\phi_{0, 1}(\tau, z)$ by $-\phi_{-1, 1/2}(\tau, w)$  and $\phi_{0, 1}(\tau, w)$ for some indices $c_1, c_2$ with $2c_1 + c_2 = c$, 
to obtain a Jacobi form $\hat Q$ of index $\begin{psmallmatrix} a & 0 \\ 0 & c \end{psmallmatrix}$ with $\mathcal{P}(\hat Q) = \psi$. 

\smallskip

Finally suppose $b = 1$, and by assumption $a, c \ge 2$. Let $\psi \in J_{k, (a+c)/2 - 1}^w$ be any Jacobi form. Since $\psi$ has index at least $1$, we can write it in the form $$\psi = \phi_{-2, 1} P + \phi_{0, 1} Q, \; \; P,Q \in \mathbb{C}[E_4, E_6, \phi_{-1, 1/2}, \phi_{0, 1}].$$ Let $\hat P, \hat Q$ be any Jacobi forms of index $\begin{psmallmatrix} a-2 & 0 \\ 0 & c-2 \end{psmallmatrix}$ with $\mathcal{P}(\hat P) = P$ and $\mathcal{P}(\hat Q) = Q$ as in the previous paragraph. Then the form $$F = \Phi_{-2, \begin{psmallmatrix} 2 & 1 \\ 1 & 2 \end{psmallmatrix}} \hat P + \Phi_{0, \begin{psmallmatrix} 2 & 1 \\ 1 & 2 \end{psmallmatrix}} \hat Q$$ satisfies $\mathcal{P}(F) = \psi$.
\end{proof}

This lemma quickly leads to a simple set of generators when both the weight and the off-diagonal index $b$ is even:

\begin{theorem}\label{thm:even} The graded ring $$\mathcal{J}_0 := J^w_{2*, \begin{psmallmatrix} 2* & 2* \\ 2* & 2* \end{psmallmatrix}} = \bigoplus_{\substack{a, b, c \in \mathbb{N}_0 \\ b \le \min(a, c)}} \bigoplus_{k \in 2 \mathbb{Z}} J_{k, \begin{psmallmatrix} 2a & 2b \\ 2b & 2c \end{psmallmatrix}}^w$$ is generated by the Eisenstein series $E_4, E_6$, the rank one Jacobi forms $$\phi_{-2, 1}(\tau, z), \; \phi_{-2, 1}(\tau, w), \; \phi_{-2, 1}(\tau, z + w),$$ $$\phi_{0, 1}(\tau, z), \; \phi_{0, 1}(\tau, w), \; \phi_{0, 1}(\tau, z+w),$$ and the products $$\phi_{-1, 2}(\tau, z) \phi_{-1, 2}(\tau, w), \; \phi_{-1, 2}(\tau, z) \phi_{-1, 2}(\tau, z+w), \; \phi_{-1, 2}(\tau, w) \phi_{-1, 2}(\tau, z+w).$$
\end{theorem}
\begin{proof}  Let $\varphi(\tau, z, w)$ be a weak Jacobi form of even weight $k$ and index $\begin{psmallmatrix} 2a & 2b \\ 2b & 2c \end{psmallmatrix}$. We will show by induction on $b$ that $\varphi$ can be written as a polynomial in the claimed generators. When $b = 0$ this follows from the classification of diagonal-index forms in Theorem \ref{thm:diag}. If $b = \min(a, c)$ then one of $\varphi(\tau, z+w, -w)$ and $\varphi(\tau, w, -z-w)$ is a Jacobi form of index with off-diagonal $b=0$. Since this change-of-variables leaves the ring generated by the Jacobi forms in this claim invariant, we may also assume that $b < \min(a, c)$. 

\smallskip

In general, using Lemma \ref{lem:surj} we can find a weak Jacobi form $F$ of index $\begin{psmallmatrix} 2(a - b) & 0 \\ 0 & 2(c - b) \end{psmallmatrix}$ with $\mathcal{P}(F) = \mathcal{P}(\varphi)$. Then the form $$\varphi(\tau, z, w) - (\phi_{0, 1}(\tau, z + w) / 12)^b F(\tau, z, w) \in J_{k, \begin{psmallmatrix} 2a & 2b \\ 2b & 2c \end{psmallmatrix}}^w$$ vanishes at every point $(z, w)$ with $z + w \in \mathbb{Z} + \tau \mathbb{Z}$ and is therefore divisible by $\vartheta(\tau, z+w)$, so $$\varphi_1(\tau, z, w) := \Big(\varphi(\tau, z, w) - (\phi_{0, 1}(\tau, z+w) / 12)^b F(\tau, z, w) \Big) / \phi_{-1, 1/2}(\tau, z+w) \in J_{k+1, \begin{psmallmatrix} 2a - 1  & 2b - 1 \\ 2b - 1 & 2c - 1 \end{psmallmatrix}}^w$$ is again a weak Jacobi form.

The image $$\mathcal{P}(\varphi_1) = \varphi(\tau, z, -z) \in J_{k+1, a + c - 2b}^w$$ is a weak Jacobi form of odd weight and integral index and is therefore a multiple of $\phi_{-1, 2}$. Since $b < \min(a, c)$, we may use Lemma \ref{lem:surj} to find a weak Jacobi form $G$ of index $\begin{psmallmatrix} 2a - 2 & 2b \\ 2b & 2c-2 \end{psmallmatrix}$ with $$\mathcal{P}(\varphi_1) / \phi_{-1, 2} = \mathcal{P}(G).$$ Then $$\mathcal{P}\Big( \varphi_1(\tau, z, w) - \phi_{-1, 1/2}(\tau, z - w) G(\tau, z, w) \Big) = 0,$$ so we can again divide by $\vartheta(\tau, z+w)$ and obtain a weak Jacobi form $$\varphi_2 := \Big( \varphi_1(\tau, z, w) - \phi_{-1, 1/2}(\tau, z-w) G(\tau, z, w) \Big) / \phi_{-1, 1/2}(\tau, z+w) \in J_{k + 2, \begin{psmallmatrix} 2a - 2 & 2b - 2 \\ 2b - 2 & 2c - 2 \end{psmallmatrix}}^w.$$ By induction, $\varphi_2$ is a polynomial as in the theorem. The claim follows by writing $$\varphi = (\phi_{0, 1}(\tau, z+w)/12)^b F(\tau, z, w) + \phi_{-1, 1/2}(\tau, z+w) \phi_{-1, 1/2}(\tau, z - w) G(\tau, z, w) + \phi_{-2, 1}(\tau, z+w) \varphi_2(\tau, z, w)$$ and using the identity $$\phi_{-1, 1/2}(\tau, z+w) \phi_{-1, 1/2}(\tau, z-w) = \frac{1}{12} \Big( \phi_{-2, 1}(\tau, z) \phi_{0, 1}(\tau, w) - \phi_{0, 1}(\tau, z) \phi_{-2, 1}(\tau, w) \Big),$$ which is easily proved using Fourier series since both sides of this equation lie in the two-dimensional space of weak Jacobi forms of weight $-2$ and index $\begin{psmallmatrix} 2 & 0 \\ 0 & 2 \end{psmallmatrix}$ which is spanned by $\phi_{-2, 1}(\tau, z)\phi_{0, 1}(\tau, w)$ and $\phi_{0, 1}(\tau, z) \phi_{-2, 1}(\tau, w)$. 
\end{proof}

The larger ring of weak Jacobi forms for even-weight, even index matrices where the off-diagonal index may be odd is already obtained by including the weak Jacobi forms of $A_2$-index:

\begin{theorem} As a $\mathcal{J}_0$-module, $$J^w_{2*, \begin{psmallmatrix} 2* & 2* + 1 \\ 2* + 1 & 2* \end{psmallmatrix}} = \bigoplus_{\substack{a, b, c \in \mathbb{N}_0 \\ b < \min(a, c)}} \bigoplus_{k \in 2 \mathbb{Z}} J_{k, \begin{psmallmatrix} 2a & 2b + 1 \\ 2b + 1 & 2c \end{psmallmatrix}}^w$$ is spanned by the forms $\Phi_{-2, \begin{psmallmatrix} 2 & 1 \\ 1 & 2 \end{psmallmatrix}}$ and $\Phi_{0, \begin{psmallmatrix} 2 & 1 \\ 1 & 2 \end{psmallmatrix}}$.
\end{theorem}
\begin{proof} Let $\varphi \in J^w_{2k, \begin{psmallmatrix} 2a & 2b+1 \\ 2b + 1 & 2c \end{psmallmatrix}}$ be a weak Jacobi form. Write $$\mathcal{P} \varphi(\tau, z) = \varphi(\tau, z, -z) \in J_{2k, a + c - 2b - 1}^w$$ in the form $$\mathcal{P}\varphi = \phi_{-2, 1} u + \phi_{0, 1}v, \; \text{where} \; u, v \in \mathbb{C}[\phi_{-2, 1}, \phi_{0, 1}].$$ (Note that the index $a + c - 2b - 1$ is strictly positive.)  Using Lemma \ref{lem:surj} we find weak Jacobi forms $F$ and $G$ of index $\begin{psmallmatrix} 2a - 2 & 2b \\ 2b & 2c - 2 \end{psmallmatrix}$ with $\mathcal{P}(F) = u$ and $\mathcal{P}(G) = v$. Moreover, $$\mathcal{P}\Big(C_1 \cdot \Phi_{-2, \begin{psmallmatrix} 2 & 1 \\ 1 & 2 \end{psmallmatrix}}\Big) = \phi_{-2, 1} \; \text{and} \; \mathcal{P}\Big(C_2 \cdot \Phi_{0, \begin{psmallmatrix} 2 & 1 \\ 1 & 2 \end{psmallmatrix}}\Big) = \phi_{0, 1}$$ for some nonzero constants $C_1, C_2$. Then $$\varphi_1 := \Big( \varphi - C_1 \Phi_{-2, \begin{psmallmatrix} 2 & 1 \\ 1 & 2 \end{psmallmatrix}} F - C_2 \Phi_{0, \begin{psmallmatrix} 2 & 1 \\ 1 & 2 \end{psmallmatrix}} G \Big) / \phi_{-1, 1/2}(\tau, z+w)$$ is a weak Jacobi form of odd weight $2k+1$ and index $\begin{psmallmatrix} 2a - 1 & 2b \\ 2b & 2c - 1 \end{psmallmatrix}$. We will now show by induction on $b$ that $\varphi$ must have a representation of the form $f \Phi_{-2, \begin{psmallmatrix} 2 & 1 \\ 1 & 2 \end{psmallmatrix}} + g\Phi_{0, \begin{psmallmatrix} 2 & 1 \\ 1 & 2 \end{psmallmatrix}}$ with $f, g \in \mathcal{J}_0$. 

\bigskip

(i) Suppose $b = 0$. Then the form $\varphi_1$ constructed above is a weak Jacobi form of diagonal index.  By Theorem \ref{thm:diag}, 
\begin{align*} 
J_{2k - 1, \begin{psmallmatrix} 2a + 1 & 0 \\ 0 & 2c + 1 \end{psmallmatrix}}^w &= \phi_{-1, 1/2}(\tau, z) \phi_{0, 3/2}(\tau, w) J^w_{2k, \begin{psmallmatrix} 2a & 0 \\ 0 & 2c - 2 \end{psmallmatrix}} + \phi_{0, 3/2}(\tau, z) \phi_{-1, 1/2}(\tau, w)J^w_{2k, \begin{psmallmatrix} 2a-2 & 0 \\ 0 & 2c \end{psmallmatrix}} \\ &\subseteq \mathcal{J}_0 \cdot  \phi_{-1, 1/2}(\tau, z) \phi_{0, 3/2}(\tau, w) + \mathcal{J}_0 \cdot \phi_{0, 3/2}(\tau, z) \phi_{-1, 1/2}(\tau, w). 
\end{align*}

Using the linear relations 
$$\phi_{-1, 1/2}(\tau, z) \phi_{0, 3/2}(\tau, w) \phi_{-1, 1/2}(\tau, z+w) \in \mathrm{Span} \Big( \Phi_{-2, \begin{psmallmatrix} 2 & 1 \\ 1 & 2 \end{psmallmatrix}}(\tau, z, w)\phi_{0,1}(\tau, w), \; \Phi_{0, \begin{psmallmatrix} 2 & 1 \\ 1 & 2 \end{psmallmatrix}} \phi_{-2, 1}(\tau, w) \Big)$$ and $$\phi_{0, 3/2}(\tau, z) \phi_{-1, 1/2}(\tau, w) \phi_{-1, 1/2}(\tau, z+w) \in \mathrm{Span} \Big( \Phi_{-2, \begin{psmallmatrix} 2 & 1 \\ 1 & 2 \end{psmallmatrix}}(\tau, z, w) \phi_{0, 1}(\tau, z), \; \Phi_{0, \begin{psmallmatrix} 2 & 1 \\ 1 & 2 \end{psmallmatrix}} \phi_{-2, 1}(\tau, z) \Big)$$ (which can be proved by showing that the space of weight $-2$ weak Jacobi forms of index $\begin{psmallmatrix} 2 & 1 \\ 1 & 4 \end{psmallmatrix}$ and $\begin{psmallmatrix} 4 & 1 \\ 1 & 2 \end{psmallmatrix}$ are two-dimensional and comparing Fourier coefficients) we obtain the desired representation for $\varphi$. 

\bigskip

(ii) Suppose $0 < b < \min(a, c) - 1$. In this case, the pullback $\mathcal{P}(\varphi_1)$ of the above $\varphi_1$ is again a weak Jacobi form of odd weight and integral index and therefore a multiple of $\phi_{-1, 2}$. By Lemma \ref{lem:surj} and its proof, we can find a weak Jacobi form $G_1\in \Phi_{-2, \begin{psmallmatrix} 2 & 1 \\ 1 & 2 \end{psmallmatrix}}\mathcal{J}_0+\Phi_{0, \begin{psmallmatrix} 2 & 1 \\ 1 & 2 \end{psmallmatrix}}\mathcal{J}_0$ of index $\begin{psmallmatrix} 2a - 2 & 2b + 1 \\ 2b + 1 & 2c - 2 \end{psmallmatrix}$ such that $$\mathcal{P}(G_1) = \mathcal{P}(\varphi_1) / \phi_{-1, 2}.$$ (Note that the condition of Lemma \ref{lem:surj} is satisfied because $b < \min(a, c) - 1$.) Then we obtain the weak Jacobi form $$\varphi_2 := \Big( \varphi_1(\tau, z, w) - \phi_{-1, 1/2}(\tau, z-w) G_1(\tau, z, w) \Big) / \phi_{-1, 1/2}(\tau, z+w) \in J^w_{2k + 2, \begin{psmallmatrix} 2a - 2 & 2b - 1 \\ 2b - 1 & 2c - 2 \end{psmallmatrix}},$$ and the claim follows by induction using the identity for $\phi_{-1, 1/2}(\tau, z+w)\phi_{-1, 1/2}(\tau, z-w)$ in the proof of Theorem \ref{thm:even}. 

\bigskip

(iii) Suppose $b = \min(a, c) - 1$; by swapping $z$ and $w$ if necessary, we may assume $b = a - 1$ and the index matrix of $\varphi$ has the form $\begin{psmallmatrix} 2a & 2a -1 \\ 2a-1 & 2c \end{psmallmatrix}$. Then $\varphi(\tau, z+w, -w)$ has index $\begin{psmallmatrix} 2a & 1 \\ 1 & 2(c-a+1) \end{psmallmatrix}$ and we obtain an expression for it using (i). The claim follows because the weak Jacobi forms of index $\begin{psmallmatrix} 2 & 1 \\ 1 & 2 \end{psmallmatrix}$ are invariant under $(z, w) \mapsto (z+w, -w)$.
\end{proof}

\section{The full ring of weak Jacobi forms}\label{sec:general}

\begin{lemma}\label{lem:surj2} Let $M = \begin{psmallmatrix} 2a & b \\ b & 2c \end{psmallmatrix}$ be an even index matrix. Suppose either that $k$ is even, or $k$ is odd and $b < 2c$. Then the map $$\mathcal{Q}: J_{k, M}^w \longrightarrow J_{k, a}^w, \; \; \varphi(\tau, z, w) \mapsto \varphi(\tau, z, 0)$$ is surjective. 
\end{lemma}
The condition on $b$ in odd weights cannot be removed in general; for example there is no weight $-1$ weak Jacobi form of index $\begin{psmallmatrix} 4 & 2 \\ 2 & 2 \end{psmallmatrix}$ that could map to $\phi_{-1, 2}$.
\begin{proof} (i) Suppose $b = 0$. Then any weak Jacobi form $f \in J_{k, a}^w$ arises as the image of the Jacobi form $f(\tau, z) (\phi_{0, 1}(\tau, w) / 12)^c$ of index $\begin{psmallmatrix} 2a & 0 \\ 0 & 2c \end{psmallmatrix}.$ 

\smallskip

(ii) Suppose $b = \min(2a, 2c)$. If $b = 2a$ then we can pass from the Jacobi form $\varphi$ to the Jacobi form $\varphi(\tau, z+w, -w)$ of index $\begin{psmallmatrix} 2a & 0 \\ 0 & 2(c - a) \end{psmallmatrix}$, preserving the value at $w = 0$, and apply (i). If $b = 2c$ and by assumption $k$ is even, then $J_{k, a}^w$ is spanned by monomials in $\phi_{0, 1}$ and $\phi_{-2, 1}$, and these arise as images of monomials in $$\phi_{0,1}(\tau, z+w), \; \phi_{-2, 1}(\tau, z+w), \; \phi_{0, 1}(\tau, z), \; \phi_{-2, 1}(\tau, z).$$

\smallskip

(iii) Suppose $0 < b < \min(2a, 2c)$. If $k$ is even then $$J_{k, a}^w = \phi_{0, 1} \cdot J_{k, a-1}^w + \phi_{-2, 1} \cdot J_{k+2, a-1}^w,$$ and by an induction on $b$ it is enough to find Jacobi forms of index $\begin{psmallmatrix} 2 & 1 \\ 1 & 2 \end{psmallmatrix}$  whose images are (a nonzero multiple of) $\phi_{0, 1}$ and $\phi_{-2, 1}$. We have seen that $\Phi_{0, \begin{psmallmatrix} 2 & 1 \\ 1 & 2 \end{psmallmatrix}}$ and $\Phi_{-2, \begin{psmallmatrix} 2 & 1 \\ 1 & 2 \end{psmallmatrix}}$ have this property. If $k$ is odd then $$J_{k, a}^w = \phi_{-1, 2} \cdot J_{k+1, a-2}^w$$ and we need to find a Jacobi form of index $\begin{psmallmatrix} 4 & 1 \\ 1 & 2 \end{psmallmatrix}$ whose image is a nonzero multiple of $\phi_{-1, 2}$. But $\Phi_{0, \begin{psmallmatrix} 3 & 2 \\ 2 & 3 \end{psmallmatrix}}(\tau, z+w, -w) \phi_{-1, 1/2}(\tau, z)$ has this property.
\end{proof}

We will combine the following technical lemmas to obtain a system of generators for the algebra of weak Jacobi forms of arbitrary rank and index. 

\begin{lemma}\label{lem:even} Let $M = \begin{psmallmatrix} 2a & b \\ b & 2c \end{psmallmatrix}$ be an even index matrix with $2a \ge 2c > b > 0$ and $M \ne \begin{psmallmatrix} 2 & 1 \\ 1 & 2 \end{psmallmatrix}$. 
\begin{itemize}
\item[(i)] If $b \le 2a - 3$, then for any $k \in \mathbb{Z}$, $$J_{2k - 1, \begin{psmallmatrix} 2a & b \\ b & 2c \end{psmallmatrix}}^w = \Phi_{0, \begin{psmallmatrix} 3 & 2 \\ 2 & 3 \end{psmallmatrix}}(\tau, z+w, -w) \phi_{-1, 1/2}(\tau, z) \cdot J_{2k, \begin{psmallmatrix}2a - 4 & b -1 \\ b-1  & 2c-2 \end{psmallmatrix}}^w + \phi_{-1, 1/2}(\tau, w) \cdot J_{2k, \begin{psmallmatrix} 2a & b \\ b & 2c - 1 \end{psmallmatrix}}^w.$$
\item[(ii)] If $b = 2a - 2$ (and necessarily $a = c$) then $$J_{2k - 1, \begin{psmallmatrix} 2a & 2a - 2 \\ 2a - 2 & 2a \end{psmallmatrix}}^w = \phi_{-1, 1/2}(\tau, 2z+w)\phi_{0,3/2}(\tau,w) \cdot J_{2k, \begin{psmallmatrix} 2a-4 & 2a-4 \\ 2a-4 & 2a-4 \end{psmallmatrix}}^w + \phi_{-1, 1/2}(\tau, z+w) \cdot J_{2k, \begin{psmallmatrix}2a - 1 & 2a - 3 \\ 2a - 3 & 2a - 1 \end{psmallmatrix}}^w.$$
\item[(iii)] If $b = 2a - 1$ then $$J_{2k - 1, \begin{psmallmatrix} 2a & 2a - 1 \\ 2a - 1 & 2a \end{psmallmatrix}}^w = \phi_{-1, 1/2}(\tau, z+w) \cdot J_{2k, \begin{psmallmatrix} 2a-1 & 2a - 2\\ 2a-2 & 2a - 1 \end{psmallmatrix}}^w.$$
\end{itemize}
\end{lemma}

\begin{proof}  Let $\varphi$ be a weak Jacobi form of index $\begin{psmallmatrix} 2a & b \\ b & 2c \end{psmallmatrix}$ and odd weight $2k+1$. 

\smallskip

(i) Suppose first that $2a \ge b + 3$. Then $\varphi(\tau, z, 0)$ is a weak Jacobi form of index $a$ and odd weight and therefore a multiple of $\phi_{-1, 2}(\tau, z)$. By assumption, $2a \ge 4$, so using Lemma \ref{lem:surj2} we can find a weak Jacobi form $F$ of weight $2k+2$ and index $\begin{psmallmatrix} 2a - 4 & b - 1 \\ b - 1 & 2c - 2 \end{psmallmatrix}$ with $$F(\tau, z, 0) = \varphi(\tau, z, 0) / \phi_{-1, 2}(\tau, z).$$ Then $$\tilde \varphi(\tau, z, w) := \varphi(\tau, z, w) - C \cdot \Phi_{0, \begin{psmallmatrix} 3 & 2 \\ 2 & 3 \end{psmallmatrix}}(\tau, z+w, -w) \phi_{-1, 1/2}(\tau, z) F(\tau, z, w)$$ (for an appropriate constant $C$) satisfies $$\tilde \varphi(\tau, z, 0) = 0,$$ so $G := \tilde \varphi(\tau, z, w) / \phi_{-1, 1/2}(\tau, w)$ is a weak Jacobi form and \[ \varphi =  \Phi_{0, \begin{psmallmatrix} 3 & 2 \\ 2 & 3 \end{psmallmatrix}}(\tau, z+w, -w) \phi_{-1, 1/2}(\tau, z) \cdot (C \cdot F) + \phi_{-1, 1/2}(\tau, w) \cdot G.\]

\smallskip

(ii) Suppose $b = 2a - 2$. Then $\mathcal{P}(\varphi) = \varphi(\tau, z, -z)$ is a weak Jacobi form of odd weight and index $2$, so it has the form $f(\tau) \phi_{-1, 2}(\tau, z)$ with $f \in \mathbb{C}[E_4, E_6]$. Since $b>0$, we have $a\geq 2$. Thus the form $$F(\tau, z, w) := f(\tau) (\phi_{0, 1}(\tau, z+w) / 12)^{a - 2} \phi_{-1, 1/2}(\tau, 2z+w)\phi_{0,3/2}(\tau,w) $$ and satisfies $\mathcal{P}(F) = \mathcal{P}(\varphi)$, so $F - \varphi$ is a multiple of $\phi_{-1, 1/2}(\tau, z+w)$. 

\smallskip

(iii) Suppose $b = 2a -1$. In this case $\mathcal{P}(\varphi)$ is a weak Jacobi form of odd weight and index $1$, so it is identically zero. Therefore $\varphi$ is divisible by $\phi_{-1, 1/2}(\tau, z+w)$.
\end{proof}

\begin{lemma}\label{lem:mixed} Let $\begin{psmallmatrix} 2a + 1 & b \\ b & 2c \end{psmallmatrix}$ be an index matrix with $0 < b < \min(2c, 2a+1)$. 
\begin{itemize}
\item[(i)] For any $k \in \mathbb{Z}$, $$J_{2k-1, \begin{psmallmatrix} 2a + 1 & b \\ b & 2c \end{psmallmatrix}}^w = \phi_{-1, 1/2}(\tau, z) \cdot J_{2k, \begin{psmallmatrix} 2a & b \\ b & 2c \end{psmallmatrix}}^w + \phi_{-1, 1/2}(\tau, w)  \cdot J_{2k, \begin{psmallmatrix} 2a+1 & b \\ b & 2c - 1 \end{psmallmatrix}}^w.$$

\item[(ii)] Suppose $b \le 2a - 2$. For any $k \in \mathbb{Z}$, 
$$J_{2k, \begin{psmallmatrix} 2a + 1 & b \\ b & 2c \end{psmallmatrix}}^w = \phi_{-1, 1/2}(\tau, w) \cdot J_{2k+1, \begin{psmallmatrix} 2a + 1 & b \\ b & 2c - 1 \end{psmallmatrix}}^w + \phi_{0, 3/2}(\tau, z) \cdot J_{2k, \begin{psmallmatrix} 2a - 2 & b \\ b & 2c \end{psmallmatrix}}^w.$$

\item[(iii)] Suppose $b = 2a - 1$. For any $k \in \mathbb{Z}$, 
\begin{align*} J_{2k, \begin{psmallmatrix} 2a + 1 & 2a-1 \\ 2a-1 & 2c \end{psmallmatrix}}^w &=  \phi_{-1, 1/2}(\tau, w) \phi_{-1, 1/2}(\tau, z+w) \cdot J_{2k + 2, \begin{psmallmatrix} 2a & 2a - 2 \\ 2a - 2 & 2c - 2 \end{psmallmatrix}}^w \\ &\quad + \phi_{-1, 1/2}(\tau, z) \cdot J_{2k + 1, \begin{psmallmatrix} 2a & 2a - 1 \\ 2a - 1 & 2c \end{psmallmatrix}}^w \\ &\quad + \Phi_{0, \begin{psmallmatrix} 3 & 2 \\ 2 & 3 \end{psmallmatrix}}(\tau, z+w, -w) \cdot J_{2k, \begin{psmallmatrix} 2a - 2 & 2a - 2 \\ 2a - 2 & 2c - 2 \end{psmallmatrix}}^w. 
\end{align*}

\item[(iv)] Suppose $b = 2a$. For any $k \in \mathbb{Z}$, 
\begin{align*} 
J_{2k, \begin{psmallmatrix} 2a+1 & 2a \\ 2a & 2c \end{psmallmatrix}}^w &= \phi_{-1, 1/2}(\tau, w) \phi_{-1, 1/2}(\tau, z+w) \cdot J_{2k + 2, \begin{psmallmatrix} 2a & 2a - 1 \\ 2a - 1 & 2c - 2 \end{psmallmatrix}}^w \\ &\quad + \phi_{-1, 1/2}(\tau, z) \cdot J_{2k + 1, \begin{psmallmatrix} 2a & 2a \\ 2a & 2c \end{psmallmatrix}}^w \\ &\quad + \Phi_{0, \begin{psmallmatrix} 3 & 1 \\ 1 & 3 \end{psmallmatrix}}(\tau, z+w, -w) \cdot J_{2k, \begin{psmallmatrix} 2a - 2 & 2a - 2 \\ 2a - 2 & 2c - 4 \end{psmallmatrix}}^w. 
\end{align*} 
\end{itemize}
\end{lemma}
\begin{proof} Let $\varphi(\tau, z, w)$ be a weak Jacobi form of index $\begin{psmallmatrix} 2a + 1 & b \\ b & 2c \end{psmallmatrix}$. 

\smallskip

(i) Suppose $\varphi$ has odd weight. Then $\varphi(\tau, z, 0)$ is a weak Jacobi form of odd weight and half-integral index, and therefore a multiple of $\phi_{-1, 1/2}$. By Lemma \ref{lem:surj2} we can find a weak Jacobi form $F$ of index $\begin{psmallmatrix} 2a & b \\ b & 2c \end{psmallmatrix}$ with $F(\tau, z, 0) = \varphi(\tau, z, 0) / \phi_{-1, 1/2}(\tau, z)$. Then $$G := (\varphi - F \cdot \phi_{-1, 1/2}(\tau, z)) / \phi_{-1, 1/2}(\tau, w)$$ is a weak Jacobi form of index $\begin{psmallmatrix} 2a + 1 & b \\ b & 2c- 1 \end{psmallmatrix}$, i.e. \[\varphi = \phi_{-1, 1/2}(\tau, z) \cdot F + \phi_{-1, 1/2}(\tau, w) \cdot G. \]

\smallskip

(ii) Since $\varphi(\tau, z, 0)$ has even weight and half-integral index, it is a multiple of $\phi_{0, 3/2}$. By Lemma \ref{lem:surj} we can find a weak Jacobi form $F$ of index $\begin{psmallmatrix} 2a - 2 & b \\ b & 2c \end{psmallmatrix}$ with $$F(\tau, z, 0) = \varphi(\tau, z, 0) / \phi_{0, 3/2}(\tau, z).$$ Then $$G := (\varphi - F \cdot \phi_{0, 3/2}(\tau, z)) / \phi_{-1, 1/2}(\tau, w)$$ is a weak Jacobi form, i.e. $\varphi =  \phi_{0, 3/2}(\tau, z) \cdot F + \phi_{-1, 1/2}(\tau, w) \cdot G.$ 

\smallskip

(iii) The form $\varphi(\tau, 0, w)$ has even weight and integer index $c$, so we can write $$\varphi(\tau, 0, w) = \phi_{0, 1}(\tau, w) f(\tau, w) + \phi_{-2, 1}(\tau, w) g(\tau, w)$$ where $f, g \in \mathbb{C}[E_4, E_6, \phi_{-2, 1}, \phi_{0, 1}]$. Let $F$ and $G$ be weak Jacobi forms of index $\begin{psmallmatrix} 2a - 2 & 2a - 2 \\ 2a - 2 & 2c - 2 \end{psmallmatrix}$ and $\begin{psmallmatrix} 2a & 2a - 2 \\ 2a - 2 & 2c - 2 \end{psmallmatrix}$, respectively, such that $$F(\tau, 0, w) = f(\tau, w) \; \text{and} \; G(\tau,  0, w) = g(\tau, w).$$ (The existence follows from Lemma \ref{lem:surj2} after swapping the roles of $z$ and $w$.) Then the form $$\tilde \varphi(\tau, z, w) := \varphi(\tau, z, w) - \Phi_{0, \begin{psmallmatrix} 3 & 2 \\ 2 & 3 \end{psmallmatrix}}(\tau, z+w, -w) F(\tau, z, w) - \phi_{-1, 1/2}(\tau, w) \phi_{-1, 1/2}(\tau, z+w) G(\tau, z, w)$$ vanishes at $z = 0$, so $\tilde \varphi(\tau, z, w) / \phi_{-1, 1/2}(\tau, z)$ is a weak Jacobi form of index $\begin{psmallmatrix} 2a & 2a - 1 \\ 2a - 1 & 2c \end{psmallmatrix}$, and the claim follows. 

\smallskip

(iv) As in (iii) we consider the form $\varphi(\tau, 0, w)$ of index $c$. Here $c > a \ge 1$, so we can write $$\varphi(\tau, 0, w) = \phi_{0, 1}(\tau, w)^2 f(\tau, w) + \phi_{-2, 1}(\tau, w) g(\tau, w)$$ with $f,g \in \mathbb{C}[E_4, E_6, \phi_{-2, 1}, \phi_{0, 1}]$. Using Lemma \ref{lem:surj2} we can find weak Jacobi forms $F$ and $G$ of index $\begin{psmallmatrix} 2a-2 & 2a - 2 \\ 2a - 2 & 2c - 4 \end{psmallmatrix}$ and $\begin{psmallmatrix}2a & 2a - 1 \\ 2a - 1 & 2c - 2 \end{psmallmatrix}$, respectively, such that $$F(\tau, 0, w) = f(\tau, w) \; \text{and} \; G(\tau, 0, w) = g(\tau, w).$$ Then $$\tilde \varphi(\tau, z, w) := \varphi(\tau, z, w) - \Phi_{0, \begin{psmallmatrix} 3 & 1 \\ 1 & 3 \end{psmallmatrix}}(\tau, z+w, -w) F(\tau, z, w) - \phi_{-1, 1/2}(\tau, w) \phi_{-1, 1/2}(\tau, z+w) G(\tau, z, w)$$ vanishes at $z=0$, so $\tilde \varphi(\tau, z, w) / \phi_{-1, 1/2}(\tau, z)$ is a weak Jacobi form of index $\begin{psmallmatrix} 2a & 2a \\ 2a & 2c \end{psmallmatrix}$.
\end{proof}

\begin{lemma}\label{lem:odd} Let $M = \begin{psmallmatrix}2a + 1 & b \\ b & 2c + 1 \end{psmallmatrix}$ be an index matrix with $0 < b < 2c + 1 \le 2a + 1$ and let $k \in \mathbb{Z}$.
\begin{itemize}
\item[(i)] If $b \leq  2a - 2$, then $$J_{2k - 1, M}^w = \phi_{0, 3/2}(\tau, z) \phi_{-1, 1/2}(\tau, w) \cdot J_{2k, \begin{psmallmatrix} 2a - 2 & b \\ b & 2c \end{psmallmatrix}}^w + \phi_{-1, 1/2}(\tau, z+w) \cdot J_{2k, \begin{psmallmatrix} 2a & b - 1 \\ b-1 & 2c \end{psmallmatrix}}^w.$$

\item[(ii)] If $b = 2a - 1$ (so $c = a$), then $$J_{2k-1, M}^w = \Phi_{0, \begin{psmallmatrix} 3 & 2 \\ 2 & 3 \end{psmallmatrix}}(\tau, z+w, -w) \phi_{-1, 1/2}(\tau, w) \cdot J_{2k, \begin{psmallmatrix} 2a - 2 & 2a - 2 \\ 2a - 2 & 2a - 2 \end{psmallmatrix}}^w + \phi_{-1, 1/2}(\tau, z+w) \cdot J_{2k, \begin{psmallmatrix} 2a & 2a-2 \\ 2a-2 & 2a \end{psmallmatrix}}^w.$$

\item[(iii)] If $b = 2a$ (so $c = a$), then $$J_{2k-1, M}^w = \phi_{-1, 1/2}(\tau, z+w) \cdot J_{2k, \begin{psmallmatrix}2a & 2a - 1 \\ 2a-1 & 2a \end{psmallmatrix}}^w.$$
\end{itemize}
\end{lemma}
\begin{proof}  Let $\varphi \in J_{2k - 1, \begin{psmallmatrix} 2a+1 & b \\ b & 2c+1 \end{psmallmatrix}}^w$. The pullback $$\mathcal{P}\varphi(\tau, z) = \varphi(\tau, z, -z)$$ is a weak Jacobi form of odd weight $2k-1$ and integral index $a-b+c+1$ and therefore a multiple of $\phi_{-1, 2}$. 

\smallskip

(i) In this case, we can apply Lemma \ref{lem:surj} to find a weak Jacobi form $F$ of weight $2k$ and index $\begin{psmallmatrix} 2a - 2& b \\ b & 2c  \end{psmallmatrix}$ with $$F(\tau, z, -z) = \varphi(\tau, z, -z) / \phi_{-1, 2}(\tau, z).$$ Then $$\tilde \varphi(\tau, z, w) := \varphi(\tau, z, w) - \phi_{0, 3/2}(\tau, z) \phi_{-1, 1/2}(\tau, w) F(\tau, z, w)$$ vanishes whenever $z+w \in \mathbb{Z} + \tau \mathbb{Z}$, so it is a multiple of $\phi_{-1, 1/2}(\tau, z+w)$. 

\smallskip

(ii) We argue as in (i) but instead take $F$ to have index $\begin{psmallmatrix} 2a - 2& 2a-2 \\ 2a-2 & 2a - 2 \end{psmallmatrix}$ with $$F(\tau, z, -z) = \varphi(\tau, z, -z) / \phi_{-1, 2}(\tau, z).$$ (This lies in $\mathbb{C}[E_4, E_6]$ as $\varphi(\tau, z, -z)$ has index $2$, so we can take $F \in \mathbb{C}[E_4, E_6] \cdot \phi_{0, 1}(\tau, z+w)^{a-1}$.) Then we consider the form $$\tilde \varphi = \varphi - \Phi_{0, \begin{psmallmatrix} 3 & 2 \\ 2 & 3 \end{psmallmatrix}}(\tau, z+w, -w)\phi_{-1, 1/2}(\tau, w) \cdot F$$ which vanishes whenever $z+w \in \mathbb{Z} + \tau \mathbb{Z}$ so it is a multiple of $\phi_{-1, 1/2}(\tau, z+w)$. 

\smallskip

(iii) In this case $\mathcal{P}\varphi$ is a weak Jacobi form of odd weight and index $1$ and therefore identically zero; so $\varphi$ is already a multiple of $\phi_{-1, 1/2}(\tau, z+w)$.
\end{proof}

\begin{lemma}\label{lem:odd2}
Let $0 < b <\min(2a+1, 2c+1)$.
\begin{itemize}
\item[(i)] If $b \ge 2$, then $$ J_{2k, \begin{psmallmatrix} 2a + 1 & b \\ b & 2c + 1 \end{psmallmatrix}}^w = \phi_{-1, 1/2}(\tau, w) \cdot J_{2k + 1, \begin{psmallmatrix} 2a+1 & b \\ b & 2c \end{psmallmatrix}}^w + \Phi_{0, \begin{psmallmatrix} 3 & 2 \\ 2 & 3 \end{psmallmatrix}} \cdot J_{2k, \begin{psmallmatrix} 2a - 2 & b-2 \\ b - 2 & 2c - 2 \end{psmallmatrix}}^w.$$
\item[(ii)] When $b = 1$, 
\begin{align*} 
J_{2k, \begin{psmallmatrix} 2a+1 & 1 \\ 1 & 2c + 1 \end{psmallmatrix}}^w &= \phi_{-1, 1/2}(\tau, z) \phi_{-1, 1/2}(\tau, w) \cdot J_{2k + 2, \begin{psmallmatrix} 2a & 1 \\ 1 & 2c \end{psmallmatrix}}^w \\ &\quad + \phi_{-1, 1/2}(\tau, z+w) \cdot J_{2k+1, \begin{psmallmatrix} 2a & 0 \\ 0 & 2c \end{psmallmatrix}}^w \\ &\quad + \Phi_{0, \begin{psmallmatrix} 3 & 1 \\ 1 & 3 \end{psmallmatrix}} \cdot J_{2k, \begin{psmallmatrix} 2a - 2 & 0 \\ 0 & 2c-2 \end{psmallmatrix}}^w. 
\end{align*}
\end{itemize}
\end{lemma}
\begin{proof} (i) Let $\varphi$ be a weak Jacobi form of even weight $2k$ and index $\begin{psmallmatrix} 2a+1 & b \\ b & 2c + 1 \end{psmallmatrix}.$ Then $\varphi(\tau, z, 0)$ is a weak Jacobi form of half-integral index $a+1/2$ and even weight and therefore a multiple of $\phi_{0, 3/2}$. By Lemma \ref{lem:surj2} we can find a weak Jacobi form $F$ of index $\begin{psmallmatrix} 2a - 2 & b-2 \\ b-2 & 2c-2 \end{psmallmatrix}$ with $$F(\tau, z, 0) = \varphi(\tau, z, 0) / \phi_{0, 3/2}(\tau, z).$$ Let $C$ be the nonzero constant with $C \cdot \Phi_{0, \begin{psmallmatrix} 3 & 2 \\ 2 & 3 \end{psmallmatrix}}(\tau, z, 0) = \phi_{0, 3/2}(\tau, z)$; then $$\varphi_1 := \Big( \varphi - C \cdot \Phi_{0, \begin{psmallmatrix} 3 & 2 \\ 2 & 3 \end{psmallmatrix}} \cdot F \Big) / \phi_{-1, 1/2}(\tau, w)$$ is a weak Jacobi form of weight $2k-1$ and index $\begin{psmallmatrix} 2a +1 & b \\ b & 2c \end{psmallmatrix}$, and the claim follows. 

\smallskip

(ii) The pullback $$\mathcal{P}\varphi(\tau, z) = \varphi(\tau, z, -z) \in J_{2k, a+c}^w$$ is an even-weight Jacobi form of index at least two, so we can write $$\mathcal{P}\varphi = \phi_{0, 1}^2 f + \phi_{-2, 1} g$$ with $f, g \in \mathbb{C}[E_4, E_6, \phi_{-2, 1}, \phi_{0, 1}].$ By Lemma \ref{lem:surj} there exist Jacobi forms $F, G$ of index $\begin{psmallmatrix} 2a - 2 & 0 \\ 0 & 2c - 2 \end{psmallmatrix}$ and $\begin{psmallmatrix} 2a & 1 \\ 1 & 2c \end{psmallmatrix}$, respectively, such that $$\mathcal{P}F = f \; \text{and} \; \mathcal{P}G = g.$$ Then the form $$\tilde \varphi(\tau, z, w) := \varphi(\tau, z, w) + \phi_{-1, 1/2}(\tau, z) \phi_{-1, 1/2}(\tau, w) G(\tau, z, w) - \Phi_{0, \begin{psmallmatrix} 3 & 1 \\ 1 & 3 \end{psmallmatrix}}(\tau, z, w) F(\tau, z, w)$$ vanishes at all points $z+w \in \mathbb{Z} + \tau \mathbb{Z}$, so $$\tilde \varphi(\tau, z, w) / \phi_{-1, 1/2}(\tau, z + w)$$ is a weak Jacobi form, and the claim follows.
\end{proof}

\begin{theorem}\label{thm:gens} The graded algebra $$\mathcal{J} := \bigoplus_{\substack{a, b, c \in \mathbb{N}_0 \\ b \le \min(a, c)}} \bigoplus_{k \in \mathbb{Z}} J^w_{k, \begin{psmallmatrix} a & b \\ b & c \end{psmallmatrix}}$$ is generated by the Eisenstein series $E_4, E_6$, the rank-one weak Jacobi forms $$f(\tau, z), \; f(\tau, w), \; f(\tau, z+w), \; \; f \in \{\phi_{-1, 1/2},\, \phi_{0, 1}, \, \phi_{0, 3/2}\},$$ the $A_2$-index forms $$\Phi_{-2, \begin{psmallmatrix} 2 & 1 \\ 1 & 2 \end{psmallmatrix}}, \; \Phi_{0, \begin{psmallmatrix} 2 & 1 \\ 1 & 2 \end{psmallmatrix}},$$ and by the odd lattice-index forms
\begin{align*}
&\Phi_{0, \begin{psmallmatrix} 3 & 2 \\ 2 & 3 \end{psmallmatrix}}(\tau, z, w),& &\Phi_{0, \begin{psmallmatrix} 3 & 2 \\ 2 & 3 \end{psmallmatrix}}(\tau,z+w,-w),& &\Phi_{0, \begin{psmallmatrix} 3 & 2 \\ 2 & 3 \end{psmallmatrix}}(\tau,z+w,-z),& \\
&\Phi_{0, \begin{psmallmatrix} 3 & 1 \\ 1 & 3 \end{psmallmatrix}}(\tau, z, w),& &\Phi_{0, \begin{psmallmatrix} 3 & 1 \\ 1 & 3 \end{psmallmatrix}}(\tau,z+w,-w),& &\Phi_{0, \begin{psmallmatrix} 3 & 1 \\ 1 & 3 \end{psmallmatrix}}(\tau,z+w,-z).& \\
\end{align*}
\end{theorem}
\begin{proof} Let $\varphi(\tau, z, w)$ be a weak Jacobi form of index $\begin{psmallmatrix} a & b \\ b & c \end{psmallmatrix}$. We will prove that $\varphi$ can be expressed as a polynomial in these forms by induction on $b$. Without loss of generality, we may assume that $a \ge c$. (Otherwise, we swap the elliptic variables $z$ and $w$. This preserves the system of generators in the claim.

\smallskip

(i) If $b = 0$, then $\varphi$ has diagonal index and the claim follows from Theorem \ref{thm:diag}.

\smallskip

(ii) Suppose $0 < b < \min(a, c)$. One of the lemmas above (depending on the parity of the weight $k$ and the parity of $a$ and $c$) applies to this index and yields a decomposition of $\varphi$ into weak Jacobi forms with lower off-diagonal index. Note that in Lemma \ref{lem:even}, $\phi_{-1, 1/2}(\tau, 2z+w)\phi_{0,3/2}(\tau,w)$ is a $\CC$-linear combination of the three forms
\begin{align*}
&\phi_{-1,1/2}(\tau,z)\Phi_{0, \begin{psmallmatrix} 3 & 1 \\ 1 & 3 \end{psmallmatrix}}(\tau,z+w,-w),&  &\phi_{-1,1/2}(\tau,z+w)\Phi_{0, \begin{psmallmatrix} 3 & 1 \\ 1 & 3 \end{psmallmatrix}}(\tau,z,w),& \\
&E_4(\tau) \Phi_{-3, \begin{psmallmatrix} 2 & 1 \\ 1 & 2 \end{psmallmatrix}}(\tau,z,w)\Phi_{-2, \begin{psmallmatrix} 2 & 1 \\ 1 & 2 \end{psmallmatrix}}(\tau,z,w).&
\end{align*}
Therefore we do not need $\phi_{-1, 1/2}(\tau, 2z+w)\phi_{0,3/2}(\tau,w)$ as a generator of $\mathcal{J}$. The claim follows by induction.

\smallskip

(iii) Suppose $b = \min(a, c)$. By applying one of the changes of variables $(z, w) \mapsto (z+w, -w)$ or $(z, w) \mapsto (w, -z-w)$ we obtain a weak Jacobi form of off-diagonal index $b = 0$ and we can apply case (i).
\end{proof}

\begin{corollary}\label{cor:ring} Let $L \subseteq \mathbb{R}^2$ be a positive-definite integral lattice. Then the ring $$\mathcal{J}_L := \bigoplus_{n=0}^{\infty} J_{*, L(n)}^w$$ is finitely-generated over $\CC$.
\end{corollary}
\begin{proof} Choose a Gram matrix for $L$ of the form $\begin{psmallmatrix} a+b & b \\ b & c+b \end{psmallmatrix}$ with $a,b,c \in \mathbb{N}_0$. If $0 \in \{a, b, c\}$ then this reduces to a statement about weak Jacobi forms of rank one by Theorem \ref{thm:diag}, so assume all of $a,b,c$ are positive.

The ring $\mathcal{J}_L$ is generated by all monomials in the generators of Theorem \ref{thm:gens} that have index $L(n)$ for some $n \in \mathbb{N}_0$. Since $$f^2 \in \mathcal{J}_0 \subseteq \mathbb{C}[E_4, E_6][h(\tau, z), h(\tau, w), h(\tau, z+w): \; h \in \{\phi_{-1, 1/2}, \phi_{0, 1}, \phi_{0, 3/2}\}]$$ for every $f \in \mathcal{J}$, we only need to consider monomials in which the rank-two generators $\Phi_{*, *}$ of Theorem \ref{thm:gens} appear with exponent at most one. Suppose $f$ is such a monomial of index $L(n)$ with $n \ge 24$ and factor $f$ in the form $f = g \cdot h$ where $g$ is a product of distinct rank-two generators and $h \in \mathbb{C}[h(\tau, z), h(\tau, w), h(\tau, z+w): \; h \in \{\phi_{-1, 1/2}, \phi_{0, 1}, \phi_{0, 3/2}\}]$. Adding the possible indices shows that $g$ has index $\begin{psmallmatrix} \alpha + \beta & \beta \\ \beta & \gamma + \beta \end{psmallmatrix}$ with $\alpha, \beta, \gamma \le 11$. Then $h$ has index $\begin{psmallmatrix} x+y & y \\ y & z+y \end{psmallmatrix}$ with $x \ge 24a - 11$, $y \ge 24b - 11$, $z \ge 24c - 11$. In particular, $h$ contains either $\phi_{0, 3/2}(\tau, z)$ to exponent at least $2a$, or $\phi_{0, 1}(\tau, z)$ to exponent at least $3a$, or $\phi_{-1, 1/2}(\tau, z)$ to exponent at least $6a$; otherwise its index would have $$x \le (2a-1) \cdot 3 + (3a - 1) \cdot 2 + (6a - 1) \cdot 1 = 18a - 6 < 24a - 11.$$ In particular, $h$ splits off a factor of index $\begin{psmallmatrix} 6a & 0 \\ 0 & 0 \end{psmallmatrix}$. By the same argument, $h$ splits off factors of index $\begin{psmallmatrix} 6b & 6b \\ 6b & 6b \end{psmallmatrix}$ and $\begin{psmallmatrix} 0 & 0 \\ 0 & 6c \end{psmallmatrix}$, and multiplying these together yields a factor of index $L(6)$. 

In particular, we have shown that $\mathcal{J}_L$ is generated by the finitely many monomials that have index $L(n)$, $n < 24$. (Of course this bound is not sharp.)
\end{proof}

A folklore conjecture states that the graded ring $J_{*,L,*}^w:=\bigoplus_{k\in \ZZ, n\in \NN_0}J_{k,L(n)}^w$ is finitely generated over $\CC$ for any positive-definite integral lattice $L$. This is known when $L$ has rank one by \cite{EZ}, and when it has rank two by the above Corollary. Theorem \ref{thm:diag} implies that it is true for $L \oplus M$ if it is true for $L$ and $M$. For Weyl-invariant Jacobi forms for irreducible root lattices other than $E_8$, Wirthm\"uller's theorem \cite{W92} provides a stronger result (the rings are polynomial algebras). For general lattices $L$ little seems to be known.

\section{The weights of generators}\label{sec:weight}

The lemmas of the previous section also make it possible to compute the Hilbert series $$\mathrm{Hilb}\, \mathcal{J}(q, r, s, t) := \sum_{a, b, c = 0}^{\infty} \sum_{k \in \mathbb{Z}} \mathrm{dim}\, J_{k, \begin{psmallmatrix} a + b & b \\ b & c + b \end{psmallmatrix}}^w q^a r^b s^c t^k$$ in closed form, and therefore the $\mathbb{C}[E_4, E_6]$-module structure of weak Jacobi forms of every rank two index. The ring $\mathcal{J}$ fits into an exact sequence of the form $$0 \longrightarrow K \longrightarrow \mathcal{J} \stackrel{\varphi \mapsto \varphi(\tau, z, 0)}{\longrightarrow} R \longrightarrow 0.$$

Assume that the above $\varphi$ has matrix index $\begin{psmallmatrix} a  & b \\ b & c \end{psmallmatrix}$. The kernel consists exactly of multiples $\phi_{-1, 1/2}(\tau, w) f(\tau, z, w)$, where $f$ has index $\begin{psmallmatrix} a & b \\ b & c - 1 \end{psmallmatrix}$. Here one has the minor issue that the inequality $c - 1 \le b$ is not satisfied if $b = c$, i.e. $f$ does not belong to $\mathcal{J}$. In any case, $K$ is quite close to $\phi_{-1, 1/2}(\tau, w) \cdot \mathcal{J}$, and due to the additivity of the Hilbert series in short exact sequences our immediate goal will be to determine the series $\mathrm{Hilb}\, R$. Here $R$ must be understood as the multigraded ring whose piece of degree $(a, b, c, k)$ is the subspace of $J_{k, (a+c)/2 - b}^w$ spanned by $\varphi(\tau, z, 0)$ for weak Jacobi forms $\varphi$ of weight $k$ and index $\begin{psmallmatrix} a & b \\ b & c \end{psmallmatrix}$.

We will compute the dimensions $\mathrm{dim}\, R_{a, b, c, k}$ by induction using the following technical lemmas.

\begin{lemma}\label{lem:evenQ} Let $k \in \mathbb{Z}$. The map $$\mathcal{Q} : J_{2k, \begin{psmallmatrix} a & b \\ b & c \end{psmallmatrix}}^w \longrightarrow J_{2k, a/2}^w, \quad \varphi(\tau, z, w) \mapsto \varphi(\tau, z, 0)$$ is surjective except for the cases $(b, c) = (0, 1)$ and $$(a, b, c) = (a, a, a+1), \; a \ge 2 \; \text{and} \; (a, b, c) = (3, 2, 2)$$ where the image is trivial, and the cases $$(b, c) = (1, 1),  \; a \ge 2,$$ where the image is exactly $\phi_{-2, 1} \cdot J_{2k+2, a/2 - 1}^w$.
\end{lemma}
\begin{proof} 
If $a = b$ then weak Jacobi forms are spanned by monomials $f(\tau, z+w) g(\tau, w)$ with $f \in J_{*, a/2}^w$ and $g \in J_{*, (c - a)/2}^w$, and setting $w = 0$ shows that this is surjective unless $(c - a) = 1$ (in which case the image is zero). Similarly, if $b = c$ then weak Jacobi forms are spanned by monomials $f(\tau, z) g(\tau, z+w)$ with $f \in J_{*, (a - c)/2}^w$, $g \in J_{*, c/2}^w$, so we can read off the image immediately in this case also. Suppose that $0 < b < \min(a, c)$. 

\smallskip

If $a, c$ are even then this claim is contained in Lemma \ref{lem:surj2}. Suppose $a$ is odd and $c$ is even. If $0 < b \le a - 3$ then $$\mathcal{Q}(J_{2k, \begin{psmallmatrix} a & b \\ b & c \end{psmallmatrix}}^w) \subseteq J_{2k, a/2}^w = \phi_{0, 3/2} \cdot J_{2k, (a-3)/2}^w = \mathcal{Q} \Big( \phi_{0, 3/2}(\tau, z) \cdot J_{2k, \begin{psmallmatrix} a - 3 & b \\ b & c \end{psmallmatrix}}^w \Big) \subseteq \mathcal{Q}\Big(J_{2k, \begin{psmallmatrix} a & b \\ b & c \end{psmallmatrix}}^w\Big)$$ showing surjectivity. When $b = a-2$ and $c$ is even, setting $w = 0$ in Lemma \ref{lem:mixed} (iii) and applying Lemma \ref{lem:surj2} yields 
\begin{align*} \mathcal{Q} \Big( J_{2k, \begin{psmallmatrix} a & a-2 \\ a-2 & c \end{psmallmatrix}}^w \Big) &= \phi_{-1, 1/2}(\tau, z) \cdot \mathcal{Q} \Big( J_{2k+1, \begin{psmallmatrix} a-1 & a -2 \\ a-2 & c \end{psmallmatrix}}^w \Big) + \Phi_{0, \begin{psmallmatrix} 3 & 2 \\ 2 & 3 \end{psmallmatrix}}(\tau, z, 0) \cdot \mathcal{Q}\Big( J_{2k, \begin{psmallmatrix} a - 3 & a - 3 \\ a - 3 & c - 2 \end{psmallmatrix}}^w \Big) \\ &= \phi_{-1, 1/2} \cdot J_{2k + 1, (a-1)/2}^w + \phi_{0, 3/2} \cdot J_{2k, (a-3)/2}^w \\ &= J_{2k, a/2}^w. 
\end{align*} 
When $b = a-1$ and $c \ge 4$ we obtain 
\begin{align*} \mathcal{Q} \Big( J_{2k, \begin{psmallmatrix} a & a-1 \\ a-1 & c \end{psmallmatrix}}^w \Big) &= \phi_{-1, 1/2}(\tau, z) \cdot \mathcal{Q} \Big( J_{2k+1, \begin{psmallmatrix} a-1 & a -1 \\ a-1 & c \end{psmallmatrix}}^w \Big) + \Phi_{0, \begin{psmallmatrix} 3 & 1 \\ 1 & 3 \end{psmallmatrix}}(\tau, z, 0) \cdot \mathcal{Q}\Big( J_{2k, \begin{psmallmatrix} a - 3 & a - 3 \\ a - 3 & c - 4 \end{psmallmatrix}}^w \Big) \\ &= \phi_{-1, 1/2} \cdot J_{2k + 1, (a-1)/2}^w + \phi_{0, 3/2} \cdot J_{2k, (a-3)/2}^w = J_{2k, a/2}^w. 
\end{align*} 
In the case $c = 2$, $b = 2$ we change variables to pass to the index matrix $\begin{psmallmatrix} a - 2 & 0 \\ 0 & 2 \end{psmallmatrix}$ to obtain surjectivity if and only if $a \ge 5$; and when $c = 2$ and $b = 0, 1$ we obtain surjectivity by a similar argument.

The same argument also applies to the cases where $a$ is even and $c$ is odd (swapping the role of $z$ and $w$ is Lemma \ref{lem:mixed}), and yields surjectivity also. Finally, when both $a$ and $c$ are odd, we use Lemma \ref{lem:odd2} to see that $\mathcal{Q}$ is surjective except when $c = 1$, in which case the image is always trivial if $b = 0$ and consists of multiples of $\phi_{-2, 1}$ when $b = 1$.
\end{proof}

\begin{lemma} Let $k \in \mathbb{Z}$. The map $$\mathcal{Q} : J_{2k+1, \begin{psmallmatrix} a & b \\ b & c \end{psmallmatrix}}^w \longrightarrow J_{2k+1, a/2}^w$$ is surjective except when $(b, c) = (0, 1)$ or $(a, b, c) = (4, 2, 2)$, in which cases the image is zero.
\end{lemma}
\begin{proof} When $b = 0$, using Theorem \ref{thm:diag} we see that $\mathcal{Q}$ is surjective unless $c = 1$, in which case all weak Jacobi forms are multiples of $\phi_{-1, 1/2}(\tau, w)$ and the image of $\mathcal{Q}$ is trivial. Similarly, when $b = c$ we can conjugate $\begin{psmallmatrix} a & c \\ c & c \end{psmallmatrix}$ to $\begin{psmallmatrix} a - c & 0 \\ 0 & c \end{psmallmatrix}$ to see that $\mathcal{Q}$ is surjective when $c \ge 3$ or $c = 1$; or when $c = 2$ and $a \ne 4$ (when $a=4$ the image is trivial). When $b = a$ we see that $\mathcal{Q}$ is surjective for all indices $c$ by a similar argument. Therefore suppose $0 < b < \min(a, c)$. We will show that $\mathcal{Q}$ is always surjective: \\
(i) When $a$ and $c$ are both even, this is contained in Lemma \ref{lem:surj2}. \\
(ii) Suppose $a$ is odd and $c$ is even. By Lemma \ref{lem:surj}, $$J_{2k-1, a/2}^w =\phi_{-1, 1/2} \cdot J_{2k, (a-1)/2}^w = \mathcal{Q} \Big( \phi_{-1, 1/2}(\tau, z) \cdot J_{2k, \begin{psmallmatrix} a-1 & b \\ b & c\end{psmallmatrix}}^w \Big) \subseteq \mathcal{Q}\Big(J_{2k-1, \begin{psmallmatrix} a & b \\ b & c \end{psmallmatrix}}^w\Big).$$
(iii) Suppose $a, c$ are both odd. Then $$J_{2k-1, a/2}^w =\phi_{-1, 1/2} \cdot J_{2k, (a-1)/2}^w = \mathcal{Q} \Big( \phi_{-1, 1/2}(\tau, z+w) \cdot J_{2k, \begin{psmallmatrix} a-1 & b-1 \\ b-1 & c-1\end{psmallmatrix}}^w \Big) \subseteq \mathcal{Q}\Big(J_{2k-1, \begin{psmallmatrix} a & b \\ b & c \end{psmallmatrix}}^w\Big)$$ by Lemma \ref{lem:surj2}.\\
(iv) Suppose $a$ is even and $c$ is odd. Then Lemma \ref{lem:mixed} (after swapping the variables $z$ and $w$) yields $$\mathcal{Q} \Big( J_{2k -1, \begin{psmallmatrix} a & b \\ b & c \end{psmallmatrix}}^w \Big) = \phi_{-1, 1/2} \cdot \mathcal{Q}\Big( J_{2k, \begin{psmallmatrix} a-1 & b \\ b & c \end{psmallmatrix}}^w \Big),$$ which equals $\phi_{-1, 1/2} \cdot J_{2k, (a-1)/2}^w= \phi_{-1, 1/2}\phi_{0, 3/2} \cdot J_{2k, a/2-2}^w = J_{2k-1, a/2}^w$ by Lemma \ref{lem:evenQ}.
\end{proof}

\begin{theorem}\label{thm:hilbert} The Hilbert series of $\mathcal{J}$ has closed form $$\mathrm{Hilb}\, \mathcal{J}(q, r, s, t) = \sum_{\substack{a, b, c \in \mathbb{N}_0 \\ k \in \mathbb{Z}}} \mathrm{dim}\, J_{k, \begin{psmallmatrix} a+b & b \\ b & c+b \end{psmallmatrix}}^w q^a r^b s^c t^k = \frac{F(q, r, s, t)}{(1 - t^4)(1 - t^6)},$$ where
\begin{align*} 
F(q, r, s, t) &= \frac{qrs t^{-2}}{(1 - qt^{-1}) (1 - rt^{-1}) (1 - st^{-1})} \\ &+ \frac{(1 - q + q^2)(1 - r + r^2)(1 - s + s^2) - qrs t^{-1} (qr + qs + rs - 2qrs) + qrs (1 - qrs)}{(1 - q)(1-qt^{-1})(1-r)(1-rt^{-1})(1-s)(1-st^{-1})} 
\end{align*}
\end{theorem}

\begin{remark} It is amusing to check that certain properties of the ring $\mathcal{J}$ are reflected in the function $F$. For example, the series $F$ is symmetric in the variables $q, r, s$, corresponding to the action of the anharmonic group on $\mathcal{J}$. Setting $t = 1$ yields $$\frac{(1 - q + q^2)(1 - r + r^2)(1 - s + s^2) + qrs(2 - q - r - s)}{(1 - q)^2(1-r)^2(1-s)^2},$$ which is the generating function of the determinant $\mathrm{det}(\begin{psmallmatrix} a+b & b \\ b & c+b \end{psmallmatrix}) = ab+ac+bc$ (up to some coefficients in exponents of determinant zero); this corresponds to the rank of $J_{*, \begin{psmallmatrix} a+b & b \\ b & c+b \end{psmallmatrix}}^w$ as a $\mathbb{C}[E_4, E_6]$-module. Also, setting $r$ to zero yields the factorization $$F(q, 0, s, t) = \frac{1 + q^3}{(1 - qt^{-1})(1 - q^2)} \cdot \frac{1 + s^3}{(1 - st^{-1})(1 - s^2)},$$ corresponding to the fact that $J^w_{*, \begin{psmallmatrix} * & 0 \\ 0 & * \end{psmallmatrix}}$ is the tensor square of $J^w_{*, *}$.
\end{remark}

\begin{proof} The structure theorem for weak Jacobi forms of rank one shows that $$ \sum_{k \in \mathbb{Z}} \mathrm{dim}\, J_{k, a/2}^w t^k = \frac{t^{-a} + t^{2 - a} + ... + 1}{(1 - t^4)(1-t^6)} = \frac{t^{-a} - t^{1-a} + t^{2-a} - t}{(1 - t)(1-t^4)(1-t^6)}$$ for every $a \in \mathbb{N}$. (When $a = 0$ this formula is incorrect; it yields the numerator $(1 - t)^2$ rather than $1 - t$.) After correcting for this we find 
\begin{align*} &\quad \sum_{k \in \mathbb{Z}} \sum_{a, b, c = 0}^{\infty} \mathrm{dim}\, J_{k, (a+b)/2}^w q^a r^b s^c t^k \\ &= \frac{1}{(1 - t)(1-t^4)(1 - t^6)} \Big(\sum_{b=0}^{\infty} r^b \sum_{c=0}^{\infty} s^c \sum_{a = 0}^{\infty} ((qt^{-1})^a (1 -t + t^2) - q^a t) + t(1-t) \sum_{c = 0}^{\infty} s^c \Big)\\ &= \frac{1}{(1 - s)(1-t) (1-t^4)(1-t^6)} \Big[ \frac{1 - t + t^2}{(1 - qt^{-1})(1 - rt^{-1})} - \frac{t}{(1 - q)(1 - r)} + t(1-t) \Big]. 
\end{align*}

Taking the exceptional cases in the above lemmas where the degree $(a, b, c, k)$ piece of $R$ does not equal $J_{k, a/2}^w$ into account, i.e. the indices $$(a, b, c) = (0, a, 1), (1, 2, 0), (2, 2, 0)$$ as well as $(b, c) = (0, 1), (1, 0)$, we obtain 
\begin{align*} 
\mathrm{Hilb}\, R &=  \frac{1}{(1 - s)(1-t) (1-t^4)(1-t^6)} \Big[ \frac{1 - t + t^2}{(1 - qt^{-1})(1 - rt^{-1})} - \frac{t}{(1 - q)(1 - r)} + t(1-t)\Big] \\ &\quad - \frac{1}{(1 - t^4)(1 - t^6)} \Big( qr^2 + q^2 r^2 t^{-1} \Big) \\ &\quad - \sum_{a = 0}^{\infty} \sum_{k \in \mathbb{Z}} \mathrm{dim}\, J_{k, a/2}^w t^k q^a s  - \sum_{a=1}^{\infty} \sum_{k \in \mathbb{Z}} \mathrm{dim}\, J_{k, a/2}^w t^k r^a s \\ &\quad - \frac{1}{(1 - t^4)(1 - t^6)} \sum_{a = 1}^{\infty} q^a r, 
\end{align*} 
where the series of weak Jacobi form dimensions is $$\sum_{a = 0}^{\infty} \sum_{k \in \mathbb{Z}} \mathrm{dim}\, J_{k, a/2}^w t^k q^a = \frac{(1 + q^3)}{(1 - t^4)(1 - t^6)(1 - qt^{-1})(1 - q^2)}.$$

The Hilbert series of the kernel $K$ is $$\mathrm{Hilb}\, K = st^{-1} \cdot \mathrm{Hilb}\, \mathcal{J} + \sum_{k \in \mathbb{Z}} \sum_{c = 0}^{\infty} \sum_{a = 0}^{\infty} \mathrm{dim}\, \mathrm{ker} \Big( \mathcal{Q} : J_{k, \begin{psmallmatrix} a + c & c \\ c & c \end{psmallmatrix}}^w \rightarrow J_{k, (a+c)/2}^w \Big) q^a r^c t^k.$$ This kernel is zero when $c = 0$ and otherwise has dimension $$\mathrm{dim}\, J_{k, \begin{psmallmatrix} a + c & c \\ c & c \end{psmallmatrix}}^w - \mathrm{dim}\, J_{k, (a+c)/2}^w$$ except for the special cases $c = 1$ and $c = 2$, $a \in \{1, 2\}$. Considering the weights of the missing $\mathbb{C}[E_4, E_6]$-module generators gives us 
\begin{align*} &\quad (1 - t^4)(1-t^6) \cdot \sum_{k \in \mathbb{Z}} \sum_{c = 0}^{\infty} \sum_{a = 0}^{\infty} \mathrm{dim}\, \mathrm{ker} \Big( \mathcal{Q} : J_{k, \begin{psmallmatrix} a + c & c \\ c & c \end{psmallmatrix}}^w \rightarrow J_{k, (a+c)/2}^w \Big) q^a r^c t^k \\ &= \sum_{c=1}^{\infty} \sum_{a = 0}^{\infty} \Big( P_{a/2}^w(t) P_{c/2}^w(t) - P_{(a+c)/2}^w(t) \Big) q^a r^c + \sum_{a=1}^{\infty} q^a r + qr^2 + q^2 r^2 t^{-1}, 
\end{align*} 
where $P_{a/2}^w(t)$ is the weak Hilbert polynomial $$P_{a/2}^w(t) = \frac{t^{-a} - t^{1-a} + t^{2 - a} - t}{1 - t}.$$ This simplifies to $$\frac{1 - q + q^2 - r - q^2 r + q^2 r^2 - r t^{-1} (1 - q - r)}{(1 - q)(1 - r)(1 - qt^{-1})(1 - rt^{-1})} + \frac{qr}{1 - q} + qr^2 + q^2 r^2 t^{-1}.$$ The Hilbert series of $\mathcal{J}$ is now determined by $$\mathrm{Hilb}\, \mathcal{J} = \mathrm{Hilb}\, K + \mathrm{Hilb}\, R.$$ Using some elementary algebraic manipulations we obtain the closed form in the claim.
\end{proof}

Theorem \ref{th:main3} follows quickly from Theorem \ref{thm:hilbert}. Namely, if $L \subseteq \mathbb{R}^2$ is an integral lattice with Gram matrix of the form $\begin{psmallmatrix} a+b & b \\ b & c+b \end{psmallmatrix}$ with $a,b,c \ge 0$, then $J_{*, L}^w$ is a free $\mathbb{C}[E_4, E_6]$-module on generators $\varphi_1,...,\varphi_n$ of weights $k_1,...,k_n$, where $t^{k_1} + ... + t^{k_n}$ is the coefficient of $q^a r^b s^c$ in the power series expansion of $F(q, r, s, t)$ about $q, r, s = 0$.

As an additional corollary, we obtain a description of weak Jacobi forms of small weight.

\begin{corollary} The minimal weight of a weak Jacobi form of index $M = \begin{psmallmatrix} a+b & b \\ b & c+b \end{psmallmatrix}$ is $$k_{\text{min}} = -\epsilon = -(a+b+c).$$ The space $J_{k_{\text{min}}, M}^w$ is always one-dimensional, spanned by the theta block $$\frac{\vartheta(\tau, z)^a \vartheta(\tau, z+w)^b \vartheta(\tau, w)^c}{\eta(\tau)^{3\epsilon}}.$$ The space $J_{k_{\text{min}} + 1, M}^w$ is nonzero if and only if $abc \ne 0$, in which case it is also one-dimensional, spanned by the form $$\Big( \frac{\vartheta'(\tau, z)}{\vartheta(\tau, z)} + \frac{\vartheta'(\tau, w)}{\vartheta(\tau, w)} - \frac{\vartheta'(\tau, z+w)}{\vartheta(\tau, z+w)} \Big) \cdot \frac{\vartheta(\tau, z)^a \vartheta(\tau, z+w)^b \vartheta(\tau, w)^c}{\eta(\tau)^{3 \epsilon}}.$$ The space $J_{k_{\text{min} + 2}, M}^w$ has dimension $$\mathrm{dim}\, J_{k_{\text{min}} + 2, M}^w = \delta_{a \ge 2} + \delta_{b \ge 2} + \delta_{c \ge 2} \le 3.$$ It is spanned by the subset of $$\frac{\phi_{0, 1}(\tau, z) \vartheta(\tau, z)^{a - 2} \vartheta(\tau, z+w)^b \vartheta(\tau, w)^c}{\eta(\tau)^{3\epsilon - 6}}, \; \frac{\phi_{0, 1}(\tau, z+w) \vartheta(\tau, z)^{a} \vartheta(\tau, z+w)^{b-2} \vartheta(\tau, w)^c}{\eta(\tau)^{3\epsilon - 6}},$$ $$\frac{\phi_{0, 1}(\tau, w) \vartheta(\tau, z)^{a} \vartheta(\tau, z+w)^{b} \vartheta(\tau, w)^{c-2}}{\eta(\tau)^{3\epsilon - 6}}$$ which are holomorphic in $z$ and $w$.
\end{corollary}
\begin{proof} From the expression of Theorem \ref{th:main3} we see that the Laurent polynomial $P(t)$ with $$\sum_{k = -\infty}^{\infty} \mathrm{dim}\, J_{k, M}^w t^k = \frac{P(t)}{(1 - t^4)(1 - t^6)}$$ has the form \[P(t) = t^{-(a+b+c)} + t^{1-(a+b+c)} \cdot \delta_{abc \ne 0} + t^{2 - (a + b + c)} \cdot (\delta_{a \ge 2} + \delta_{b \ge 2} + \delta_{c \ge 2}) + O(t^{3 - (a + b + c)}). \qedhere \]
\end{proof}

\bigskip

\noindent
\textbf{Acknowledgements} 
H. Wang would like to thank Kaiwen Sun for valuable discussions, and he is grateful to Max Planck Institute for Mathematics in Bonn for its hospitality and financial support.

\bibliographystyle{plainnat}
\bibliofont
\bibliography{\jobname}

\end{document}